\documentclass[a4paper, 12pt]{article}
\usepackage[utf8]{inputenc}
\usepackage[T1]{fontenc}
\usepackage{amsthm,amssymb,amsmath, amsfonts}
\usepackage{graphicx,algorithmic,algorithm}
\usepackage[]{datetime}
\usepackage{aliascnt}
\usepackage{lineno}

\usepackage{hyperref}

\usepackage{enumerate}

\usepackage{xcolor}
\definecolor{colorD}{RGB}{168,220,82}
\definecolor{colorC}{RGB}{58,221,255}
\definecolor{colorM}{HTML}{F0A202}
\definecolor{colorJ}{RGB}{255,244,58}

\usepackage{tikz}
\usetikzlibrary{calc}
\usetikzlibrary{backgrounds}
\usetikzlibrary{decorations.pathreplacing}
\tikzset{black node/.style={draw, circle, fill = black, minimum size = 5pt, inner sep = 0pt}}
\tikzset{white node/.style={draw, circle, fill = white, minimum size =
    5pt, inner sep = 0pt}}
\tikzset{bag/.style={draw, circle, fill = white, minimum size = 0.75cm, inner sep = 0pt}}
\tikzset{smallbag/.style={draw, circle, fill = white, minimum size =
    0.33cm, inner sep = 0pt, thin}}
\tikzset{medbag/.style={draw, circle, fill = white, minimum size = 0.5cm, inner sep = 0pt, thin}}
\tikzset{normal/.style = {draw=none, fill = none, rectangle}}

\hypersetup{
  pdftitle = {A Menger-like property of tree-cut width},
  pdfauthor=  {Archontia C. Giannopoulou, O-joung Kwon, Jean-Florent Raymond, and Dimitrios Thilikos},
  colorlinks = true,
  urlcolor = black!30!pink,
  linkcolor = black!30!red,
  citecolor = black!30!green,
  menucolor = {red}
  filecolor = {cyan},
  anchorcolor = {black}
  urlcolor = {magenta},
  runcolor = {cyan},
 linkbordercolor = {white}
}

\newcommand{\mynewtheorem}[2]{
  \newaliascnt{#1}{dummy}
  \newtheorem{#1}[#1]{#2}
  \aliascntresetthe{#1}
  \expandafter\def\csname #1autorefname\endcsname{#2}
}

\theoremstyle{plain}
  \mynewtheorem{theorem}{Theorem}
  \mynewtheorem{proposition}{Proposition}
  \mynewtheorem{corollary}{Corollary}
  \mynewtheorem{lemma}{Lemma}
\mynewtheorem{sublemma}{Sublemma}
\theoremstyle{definition}
  \mynewtheorem{definition}{Definition}
\theoremstyle{remark}
\mynewtheorem{remark}{Remark}
\mynewtheorem{conjecture}{Conjecture}
\mynewtheorem{question}{Question}       
\mynewtheorem{observation}{Observation}
\mynewtheorem{problem}{Problem}

\newtheoremstyle{claimstyle}
  {\topsep}   
  {\topsep}   
  {\itshape}  
  {}       
  {} 
  {\underline{.}}         
  {5pt plus 1pt minus 1pt} 
  {{\underline{\thmname{#1}\thmnumber{ #2}}}}          
  \theoremstyle{claimstyle}
  \mynewtheorem{claim}{Claim}

  \DeclareMathOperator{\dist}{{\sf dist}}
\DeclareMathOperator{\adh}{{\sf adh}}
\DeclareMathOperator{\tcw}{{\bf tcw}}
\DeclareMathOperator{\tw}{{\bf tw}}
\DeclareMathOperator{\width}{{\sf width}}

\newcommand{\X}{\mathcal{X}}

\newcommand{\N}{\mathbb{N}}

\newcommand{\itemref}[1]{\hyperref[#1]{(\ref*{#1})}}

\begin{document}

\title{A Menger-like property of tree-cut width\footnote{A preliminary
  version of this paper appeared in the proceedings of the
  36\textsuperscript{th} Symposium on Theoretical Aspects of
  Computer Science (STACS '19)~\cite{Giannopoulou2018lean}.\newline
  This work has been supported by the ERC consolidator grant DISTRUCT-648527 (A.~C.~Giannopoulou and J.-F.\ Raymond), the Institute for Basic Science (IBS-R029-C1) (O.\ Kwon), the National Research Foundation of Korea (NRF) grant funded by the Ministry of Education No.\ NRF-2018R1D1A1B07050294 (O.\ Kwon), the projects ``ESIGMA" (ANR-17-CE23-0010) and ``DEMOGRAPH'' (ANR-16-CE40-0028) (D.\ M.\ Thilikos), and by the Research Council of Norway and the French Ministry of Europe and Foreign Affairs, via the Franco-Norwegian project PHC AURORA 2019 (D.~M.~Thilikos). At the time this research was carried out, J.-F. Raymond was affiliated to the Logic and Semantics research group at Technische Universit\"at Berlin, Germany.}}

\author{
  Archontia C. Giannopoulou\footnote{Department of Informatics and Telecommunications, National and Kapodistrian University of Athens, Greece.}
  \\
  O-joung Kwon\footnote{Department of Mathematics, Incheon National University, South Korea and Discrete Mathematics Group, Institute~for~Basic~Science~(IBS), Daejeon,~South~Korea.}
  \\
  Jean-Florent Raymond\footnote{CNRS, LIMOS, Université Clermont Auvergne, France.}
  \\
  Dimitrios M. Thilikos\thanks{LIRMM, Univ Montpellier, CNRS, Montpellier, France.}
}

\date{\empty}

\maketitle

\begin{abstract}
In 1990, Thomas proved that every graph admits a tree decomposition of minimum width that additionally satisfies a certain vertex-connectivity condition called \emph{leanness} [A~Menger-like property of tree-width: The finite case. \textit{Journal of Combinatorial Theory, Series B}, 48(1):67 – 76, 1990]. This result had many uses and has been extended to several other decompositions.
 In this paper, we consider tree-cut decompositions, that have been introduced by Wollan as a possible edge-version of tree decompositions [The structure of graphs not admitting a fixed immersion. \textit{Journal of Combinatorial Theory, Series B}, 110:47 – 66, 2015]. We show that every graph admits a tree-cut decomposition of minimum width that additionally satisfies an edge-connectivity condition analogous to Thomas' leanness.
\end{abstract}

\section{Introduction}

The notion of treewidth is a cornerstone of the theory of graph minors of Robertson and Seymour. Formally, a \emph{tree decomposition} of a graph $G$ is a pair $(T, \X)$ where $T$ is a tree and $\X = \{X_t \subseteq V(G),\ t \in V(T)\}$ is a collection of vertex sets, called \emph{bags}, with the following properties:
\begin{enumerate}
\item $\bigcup_{t \in V(T)} X_t = V(G)$;
\item every edge of $G$ belongs to some bag in $\X$; and
\item for every $u \in V(G)$, the set $\{t \in V(T),\ u \in X_t\}$ induces a connected subgraph of~$T$.
\end{enumerate}
The \emph{width} of $(T,\X)$ is $\max_{t \in V(T)} |X_t| - 1$ and the \emph{treewidth} of $G$, that we denote by $\tw(G)$, is defined as the minimum width of a tree decomposition of~$G$.

When writing \cite{robertson1990graph} (see {\cite[Theorem~2]{THOMAS199067}} and the introduction of \cite[Section~5]{robertson1990graph}), Robertson and Seymour proved that there exist tree decompositions of ``small'' width that satisfy a certain connectivity condition here called \emph{linkedness property}.

\begin{theorem}[\cite{robertson1990graph}]\label{th:rsexp}
  Every graph $G$ admits a tree decomposition $(T, \{X_t\}_{t \in V(T)})$ of width less than $3\cdot 2^{\tw(G)}$ such that the following holds:
  \begin{description}
  \item[\emph{Linkedness property:}] for every $k\in \N$ and $a,b \in V(T)$, either there are $k$ vertex-disjoint paths linking $X_a$ to $X_b$, or there is a node $c$ on the path of $T$ between $a$ and $b$ such that $|X_c| < k$.
  \end{description}
\end{theorem}

That is, while Menger's theorem \cite{menger1927allgemeinen} states the existence of a set of less than $k$ vertices disconnecting $X_a$ from~$X_b$, \autoref{th:rsexp} additionally guarantees that such vertices that can be found as a bag of size less than $k$ of the tree decomposition.
The exponential bound in \autoref{th:rsexp} has been subsequently improved by Thomas to its optimal value. In fact, Thomas showed a stronger property called~\emph{leanness}.

\begin{theorem}[\cite{THOMAS199067}]\label{th:thomas}
  Every graph $G$ admits a tree decomposition $(T, \{X_t\}_{t \in V(T)})$ of width $\tw(G)$ such that the following holds:
  \begin{description}
  \item[\emph{Leanness property:}] for every $k\in \N$, $a,b \in V(T)$, $A\subseteq X_a$, and $B\subseteq X_{b}$ such that $|A| = |B| = k$, either there are $k$ vertex-disjoint paths linking $A$ to $B$, or there is a node $c$ on the path of $T$ between $a$ and $b$ such that $|X_c| < k$.
  \end{description}
\end{theorem}

A simplified proof of \autoref{th:thomas} was then found by Diestel and Bellenbaum~\cite{bellenbaum2002two}.
The aim of Robertson and Seymour was to use \autoref{th:rsexp} as an ingredient in their proof that graphs of bounded treewidth are well-quasi-ordered by the minor relation~\cite{robertson1990graph}.
Since then, the notions of leanness and linkedness have been extensively studied and extended to several
different width parameters such as  $\theta$-tree-width \cite{CarmesinDHH14, geelen2016generalization}, pathwidth \cite{Lagergren98}, directed path-width \cite{KimS15},  DAG-width~\cite{Kintali14},  rank-width~\cite{OUM200579}, linear-rankwidth~\cite{KanteK14}, profile- and block-width \cite{ERDE2018114}, matroid treewidth \cite{GeelenGW02bran, AZZATO2011123, ERDE2018114} and matroid branchwidth~\cite{GeelenGW02bran}.
They have important applications, for instance in order to bound the size of obstructions for certain classes of graphs~\cite{seymour1993bound, Lagergren98, geelen2002branch, GiannopoulouPRTW2016}, in well-quasi-ordering proofs \cite{oum2008rank, liu2014graph, GEELEN2002270}, in extremal graph theory \cite{oporowski1993typical, chudnovsky2011edge}, and for algorithmic purpose~\cite{2018arXiv181006864C}. We refer to \cite{ERDE2018114} for a unified introduction to lean decompositions.

In this paper, we show that a similar leanness property holds for \emph{tree-cut width}.
Tree-cut width is a graph invariant introduced by Wollan
in~\cite{Wollan201547} and defined via graph decompositions called \emph{tree-cut decompositions}.
Several results are supporting the claim that tree-cut width would be
the \emph{right} parameter for studying graph immersions. For instance,
there is an analog to the Grid-minor Exclusion Theorem of Robertson
and Seymour \cite{RobertsonS86GMV} in the setting of immersions~\cite{Wollan201547}.
Also, tree-cut decompositions can be used for dynamic programming in
the same way as tree decompositions do, for certain algorithmic
problems that cannot be tackled under the bounded-treewidth
framework~\cite{ganian2015, KimOPST18, Giannopoulou2016linear}.
Therefore, we expect that this invariant
will play a central role in the flourishing theory of graph
immersions.

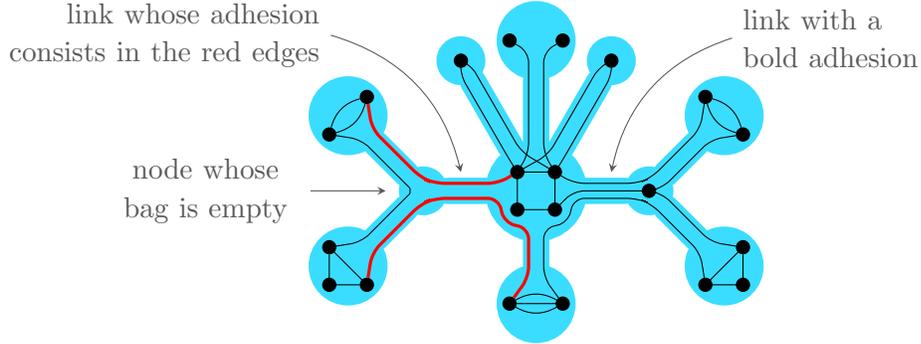
\begin{figure}[h]
  \centering
  \begin{tikzpicture}[scale = 1] 

    \begin{scope}[xshift = 8cm, yshift = -0.5cm]
      \begin{scope}[every node/.style ={black node, fill = colorC}, every path/.style = {line width = 0.35cm, color = colorC, line cap=round}]
        \draw (0,0) node[minimum size = 1cm] (middle) {};
        \draw (120:2) node[minimum size = 0.3cm] (top1) {} (90:2) node[minimum size = 0.7cm] (top2) {} (60:2) node[minimum size = 0.3cm] (top3) {};
        \draw (top1) -- (middle) (top2) -- (middle) (top3) -- (middle);
        \draw (-1.5, 0) node[minimum size = 0.3cm] (leftm) {};
        \draw (-2.5, 1) node[minimum size = 0.7cm] (leftt) {};
        \draw (-2.5, -1) node[minimum size = 0.7cm] (leftb) {};
        \draw (middle) -- (leftm) (leftb) -- (leftm) -- (leftt);
        \draw (1.5, 0) node[minimum size = 0.3cm] (rightm) {};
        \draw (2.5, 1) node[minimum size = 0.7cm] (rightt) {};
        \draw (2.5, -1) node[minimum size = 0.7cm] (rightb) {};
        \draw (middle) -- (rightm) (rightb) -- (rightm) -- (rightt);
        \draw (0, -1.5) node[minimum size = 0.7cm] (bottom) {};
        \draw (middle) -- (bottom);
      \end{scope}
      \begin{scope}[every node/.style= black node]
        \draw (0.25,-0.25) node (n14) {} -- ++(-0.5,0) node (n15) {} -- ++(0,0.5) node (n16) {} --++ (0.5,0) node (n17) {} -- cycle;
        \draw ($(leftt.center) + (0.25, 0.25)$) node (n1) {}
        ($(leftt.center) + (-0.25, -0.25)$) node (n2) {};
        \draw (n1) to[bend right] (n2) (n1) to[bend left] (n2); 
        \draw ($(leftb.center) + (-0.25, 0.25)$) node (n3) {} --
        ($(leftb.center) + (0.25, -0.25)$) node (n5) {};
        \draw (n3) -- ++(0,-0.5) node (n4) {} -- (n5); 
        \draw ($(bottom.center) + (-0.3536,0)$) node (n6) {} --
        ($(bottom.center) + (0.3536,0)$) node (n7) {};
        \draw (n6) to[bend left] (n7) (n6) to[bend right] (n7); 
        \draw ($(rightb.center) + (-0.25, -0.25)$) node (n8) {} --
        ($(rightb.center) + (0.25, 0.25)$) node (n10) {};
        \draw (n8) -- ++(0.5,0) node (n9) {} -- (n10); 
        \draw (rightm.center) node (n13) {}
        ($(rightt.center) + (0.25,-0.25)$) node (n11) {}
        ($(rightt.center) + (-0.25,0.25)$) node (n12) {};
        \draw (n11) to[bend left] (n12) (n11) to[bend right] (n12); 
        \draw (top1) node (t1) {};
        \draw (top2) ++(-0.3536,0) node (t2a) {};
        \draw (top2) ++(0.3536,0) node (t2b) {};
        \draw (top3) node (t3) {};
        \draw[rounded corners, very thick, color = red] (n16) --
        ($(middle.center) + (-180:0.5) + (0,0.1)$) --
        ($(leftm.center) + (0:0.15) + (0,0.1)$) --
        ($(leftm.center) + (135:0.15) + (45:0.1)$) --
        ($(leftt.center) + (-45:0.35) + (45:0.1)$) --
        (n1);
        \draw[rounded corners] (n2) --
        ($(leftt.center) + (-45:0.35) + (225:0.1)$) --
        ($(leftm.center) + (135:0.15) + (225:0.1)$);
        \draw ($(leftm.center) + (135:0.15) + (225:0.1)$) --
        ($(leftm.center) + (225:0.15) + (135:0.1)$);
        \draw[rounded corners] ($(leftm.center) + (225:0.15) + (135:0.1)$) -- 
        ($(leftb.center) + (45:0.35) + (135:0.1)$) --
        (n3);
        \draw[rounded corners, very thick, color = red] (n5) --
        ($(leftb.center) + (45:0.35) + (-45:0.1)$) --
        ($(leftm.center) + (225:0.15) + (-45:0.1)$) --
        ($(leftm.center) + (0:0.15) + (0,-0.1)$) --
        ($(middle.center) + (180:0.5) + (0,-0.1)$) --
        ($(middle.center) + (225:0.6)$) --
        ($(middle.center) + (-90:0.5) + (-0.1,0)$) --
        ($(bottom.center) + (90:0.35) + (-0.1,0)$) --
        (n6);
        \draw[rounded corners] (n7) --
        ($(bottom.center) + (90:0.35) + (0.1,0)$) --
        ($(middle.center) + (-90:0.5) + (0.1,0)$) --
        ($(middle.center) + (-45:0.6)$) --
        ($(middle.center) + (0:0.5) + (0,-0.1)$) --
        ($(rightm.center) + (180:0.15) + (0,-0.1)$) --
        ($(rightm.center) + (-45:0.15) + (225:0.1)$) --
        ($(rightb.center) + (135:0.35) + (225:0.1)$) --
        (n8);
        \draw[rounded corners] (n10) --
        ($(rightb.center) + (135:0.35) + (45:0.1)$) --
        ($(rightm.center) + (-45:0.15) + (45:0.1)$) --
        (n13) --
        ($(rightm.center) + (45:0.15) + (-45:0.1)$) --
        ($(rightt.center) + (225:0.35) + (-45:0.1)$) --
        (n11);
        \draw[rounded corners] (n12) --
        ($(rightt.center) + (225:0.35) + (135:0.1)$) --
        ($(rightm.center) + (45:0.15) + (135:0.1)$) --
        ($(rightm.center) + (180:0.15) + (0,0.1)$) --
        ($(middle.center) + (0:0.5) + (0,0.1)$) --
        (n17);
        \draw[rounded corners] (n13) --
        ($(rightm.center) + (180:0.15)$) --
        ($(middle.center) + (0:0.5)$) --
        (n14);
        \draw[rounded corners] (n16) --
        ($(middle.center) + (120:0.5) + (210:0.1)$) --
        ($(top1.center) + (-60:0.15) + (210:0.1)$) --
        (t1) --
        ($(top1.center) + (-60:0.15) + (60:0.1)$) --
        ($(middle.center) + (120:0.5) + (60:0.1)$) --
        (n17);
        \draw[rounded corners] (n16) --
        ($(middle.center) + (90:0.5) + (180:0.1)$) --
        ($(top2.center) + (-90:0.15) + (180:0.1)$) --
        (t2a) (t2b) --
        ($(top2.center) + (-90:0.15) + (0:0.1)$) --
        ($(middle.center) + (90:0.5) + (0:0.1)$) --
        (n17);
        \draw[rounded corners] (n16) --
        ($(middle.center) + (60:0.5) + (150:0.1)$) --
        ($(top3.center) + (-120:0.15) + (150:0.1)$) --
        (t3) --
        ($(top3.center) + (-120:0.15) + (-30:0.1)$) --
        ($(middle.center) + (60:0.5) + (-30:0.1)$) --
        (n17);
      \end{scope}
      \begin{scope}[color = black!66]

      \draw[-stealth] (-3, 0) node[normal, anchor = east, text width = 2.5cm,
      align = center]  {\small node whose\\ bag is empty} -- (-2,0);
      \draw (-5.5, 1.5) node[normal, anchor = south, text width = 5.25cm, align
      = right] (cap2) {\small link whose adhesion\\ consists in the red edges};
      \draw[-stealth] (cap2.east) to[bend left] (-1, 0.25);
      
      \draw (4, 1.5) node[normal, anchor = south, text width = 2.5cm, align
      = left] (cap3) {\small link with a bold adhesion};
      \draw[-stealth] (cap3.west) to[bend right] (1, 0.25);
    \end{scope}
    \end{scope}

    \end{tikzpicture}
  \caption[Example of a tree-cut decomposition.]{Representation of a
    tree-cut decomposition. The tree of
    the decomposition is depicted in blue and the graph is draw in
    black on top of it to specify which vertices (resp. edges) belong
    to which bags (resp.\ adhesions).}
  \label{fig:tcex}
\end{figure}

Formally, a \emph{tree-cut decomposition} of
a graph $G$ is a pair $\mathcal{D} = (T, \X)$ where $T$ is a tree and
$\X = \{X_t \subseteq V(G),\ t \in V(T)\}$ is a collection of
disjoint vertex sets, called \emph{bags}, with the property that
$\bigcup_{t \in V(T)} X_t = V(G)$.\footnote{In other words, $\{X_t,\ t \in
  V(T)\}$ is a partition of $V(G)$ plus possibly some empty sets.}
See \autoref{fig:tcex} for an example.
To avoid confusion with the vertices or edges of $G$, we respectively
use the synonyms \emph{nodes} and \emph{links} when we refer to the
vertices and edges of the tree of a tree-cut decomposition.
Let $uv$ be a link of $T$, let $T_{uv}$ and $T_{vu}$ be the two connected
components of $T - uv$, let $X^T_{uv} = \bigcup_{t \in V(T_{uv})}X_t$ and symmetrically for~$X^T_{vu}$.
The \emph{adhesion} $\adh_{\mathcal{D}}(uv)$ of the link $uv$ is defined as the set
of edges of $G$ with one endpoint in $X^T_{uv}$ and the other one in
$X^T_{vu}$. We drop the superscript when it is clear from the context.
We say that an adhesion is \emph{bold} if it has size more than two.
Then the width of the decomposition $\mathcal{D}$ is defined as:
\begin{align*}
  \width(\mathcal{D}) =& \max \left \{ \max_{e \in E(T)} |\adh(e)|,\right.\\
&\hphantom{\max \left \{ \right.}\left.    \max_{t \in V(T)} \left (|X_t| + |\{ t' \in N_T(t),\ \adh(tt')\
      \text{is bold} \}| \right ) \right \},
\end{align*}
where $N_T(t)$ denotes the set of nodes of $T$ that are
adjacent to~$t$.
The tree-cut width of $G$ is the defined as the minimum width of a
tree-cut decomposition of it.
We note that this definition differs from the original definition of
Wollan in~\cite{Wollan201547}, however the two definitions have been proved to be equivalent
in~\cite{Giannopoulou2016linear}.
Our definition of leanness for tree-cut decompositions
is a transposition to the edge setting of the leanness notion of Thomas.

\begin{definition}[leanness property for tree-cut decompositions]\label{boundless}
  A tree-cut decomposition $(T, \X)$ is said to be \emph{lean} if for
  every $k\in \N$, every $a,b \in E(T)$, and every $A\subseteq
  \adh(a), B \subseteq \adh(b)$ such that $|A|=|B|=k$, one of the following holds:
\begin{itemize}
\item there are $k$ edge-disjoint paths linking $A$ to $B$; or
\item there is a link $c$ on the path of $T$ between $a$ and $b$ such that $|\adh(c)|<k$.
\end{itemize}
\end{definition}

Notice that Thomas' notion of leanness for tree decompositions relates vertex-disjoint
paths to vertex-separators given by the decomposition, while ours links
edge-disjoint paths to edge-separators.
A related notion of linkedness has been previously studied
in \cite{GiannopoulouPRTW2016} in the simpler setting of cutwidth
orderings.
Our main result is the following.

\begin{theorem}\label{pluralism}
  Every graph $G$ admits a tree-cut decomposition of width $\tcw(G)$ that is~lean.
\end{theorem}

This result can be used to give explicit upper-bounds on the size of the
immersion-obstructions of graphs of bounded tree-cut
width\footnote{We refer here, for every $k\in \N$, to the immersion-minimal graphs that
  have tree-cut width more than~$k$.}, a result
that we postpone to a future paper (see \cite{Giannopoulou2018lean}
for an extended abstract containing both results).

\section{Preliminaries}
\label{trickster}

Given two integers $a,b$ we denote by $[a,b]$ the set $\{a,\ldots,b\}$ and by $[a]$ the set $\{1,\ldots,a\}$.

\paragraph*{Graphs}
Unless otherwise specified, we follow standard graph theory
terminology; see e.g.~\cite{Diestel05grap}.
All graphs considered in this paper are finite, undirected, without loops, and may have
multiple edges.
The vertex set of a graph $G$ is denoted by $V(G)$ and its multiset
of edges by~$E(G)$.
For a subset of vertices $S\subseteq V(G)$, $G-S$ is the induced subgraph $G[V(G) \setminus S]$.
For a subset $F\subseteq E(G)$ of edges, $G-F$ is the subgraph $(V(G), E(G) \setminus F$). The subgraph of $G$ \emph{induced} by $F$ has the set of endpoints of edges in $F$  as vertex set and $F$ as edge set.

A path of $G$ \emph{links} two edges of $G$ if it starts with one and ends with the other.
Given two sets $A$ and $B$ of edges of $G$, we say that a path $P$ \emph{links} $A$ and $B$ if
it starts with an edge of $A$, ends with an edge of $B$, and none of its internal edges belong to~$A \cup B$. In particular, the path reduced to a single edge $e \in A \cap B$ links $A$ and~$B$.

A \emph{cut} in a graph $G$ is a set $F \subseteq E(G)$ such that $G -
F$ has more connected components than~$G$.
If $A,B\subseteq E(G)$, we say that $F$ is an $(A,B)$-cut if no path
links $A$ and $B$ in $G-F$. In particular, $A \cup B$ is an
$(A,B)$-cut.

For two subsets $X,Y \subseteq V(G)$, we denote by $E_G(X, Y)$ the set
of all edges $xy \in E(G)$ for which $x\in X$ and $y \in Y$.
For $k\in \N$, a graph is said to be \emph{$k$-edge-connected} if it
has no cut on (strictly) less than $k$ edges.

\paragraph*{Tree-cut decompositions}

Let $G$ be a graph and let $(T, \X=\{X_t\}_{t \in V(T)})$ be a
tree-cut decomposition of $G$, as defined in the introduction.
For any nodes $u,v \in V(T)$, we denote by $uTv$ the (unique) path of $T$ with
endpoints $u$ and $v$. Similarly, if $e,f \in E(T)$, we denote by
$eTf$ the (unique) path of $T$ starting with $e$ and ending with~$f$.
Notice that if $G$ is 3-edge-connected, then every link of $T$ has a
bold adhesion. In this case the definition of the width of $(T,\X)$
can be simplified to
\begin{equation}\label{eq:tcwdeg}
\max \left \{ \max_{e \in E(T)} |\adh(e)|,\
    \max_{t \in V(T)} \left (|X_t| + \deg_T(t) \right ) \right \}.
\end{equation}
When a tree-cut decomposition is not lean (see \autoref{boundless}),
this is witnessed by what we call a non-leanness certificate, defined
as follows.

\begin{definition}[non-leanness certificate]\label{forbidden}
Let $(T, \X)$ be a tree-cut decomposition of a graph~$G$.
A \emph{non-leanness certificate} for $(T,\X)$ is a quintuple
$(k,a,b,A,B)$ where $k\in \N_{\geq 1}$, $a$ and $b$ are links
of $T$ and $A$ and $B$ are sets of edges of $G$ where $A\subseteq \adh(a), B \subseteq \adh(b)$ and $|A|=|B|=k$, such that the following two conditions hold:
\begin{enumerate}[(A)]
\item \label{distanced} there is no collection of $k$ edge-disjoint paths linking $A$ to $B$ and
\item \label{commanded} every link $e$ in $aTb$ satisfies $|\adh(e)|\geq k$. 
\end{enumerate}
A \emph{minimal non-leanness certificate} of $G$ is a non-leanness
certificate of the form $(k, a, b, A,B)$ (for some $k,a,b,A,B$ as
above) such that, among all non-leanness certificates of $G$, the value
of $k$ is minimum and, subject to that, the distance between $a$ and
$b$ is minimum (possibly $a = b$).
\end{definition}

In \autoref{sec:thinghood} we show how a non-leanness certificate can
be used in order to gradually improve a tree-cut decomposition towards leanness.
We now define the operation that we use for these improvement steps.

\begin{figure}[h]
  \centering
  \begin{tikzpicture}[scale = 0.45, every node/.style = medbag]
  \begin{scope}
    \draw (0,0) node (b0) {}
    -- ++(2,0) node (b2) {}
    node[normal,midway, yshift = 0.25cm] {$a$}
    -- ++(2,0) node {}
    -- ++(2,0) node {}
    -- ++(2,0) node (b1) {}
    node[normal,midway, yshift = 0.25cm] {$b$}
    (b1) -- ++(1,1) node {}
    (b1) -- ++(1,-1) node {}
    (b1) -- ++(2,0) node {}
    (b0) -- ++(-1,1) node {}
    (b0) -- ++(-1,-1) node {}
    (b2) -- ++(1,1) node {};
  \end{scope}
  
  \begin{scope}[xshift = 15cm, yshift = 2.5cm]

    \draw[every node/.style = {medbag, fill = colorC}]
    (0,0) node (b0) {}
    -- ++(2,0) node (b2) {}
    node[normal,midway, yshift = 0.25cm] {$a_1$}
    -- ++(2,0) node {}
    -- ++(2,0) node {}
    -- ++(2,0) node[label=-90:$s_1$] (uf) {{\small $\emptyset$}}
    node[midway,normal, yshift = 0.25cm] {$b_1$}
    -- ++(2,0) node (b1) {}
    node[midway,normal, yshift = 0.25cm] {$b_1'$}
    (b1) -- ++(1,1) node {}
    (b1) -- ++(1,-1) node {}
    (b1) -- ++(2,0) node {}
    (b0) -- ++(-1,1) node {}
    (b0) -- ++(-1,-1) node {}
    (b2) -- ++(1,1) node {};
    \draw[yshift = -5cm, every node/.style = {medbag, fill = colorD}]
    (0,0) node (b0) {}
    -- ++(2,0) node[label=90:$s_2$] (ue) {{\small $\emptyset$}}
    node[midway,normal, yshift = -0.25cm] {$a_2'$}
    -- ++(2,0) node (b2) {}
    node[midway,normal, yshift = -0.25cm] {$a_2$}
    -- ++(2,0) node {}
    -- ++(2,0) node {}
    -- ++(2,0) node (b1) {}
    node[normal,midway, yshift = -0.25cm] {$b_2$}
    (b1) -- ++(1,1) node {}
    (b1) -- ++(1,-1) node {}
    (b1) -- ++(2,0) node {}
    (b0) -- ++(-1,1) node {}
    (b0) -- ++(-1,-1) node {}
    (b2) -- ++(1,1) node {};
    \draw (uf) -- (ue);
  \end{scope}
  \draw[-stealth, very thick] (11.5,0) -- ++(1,0);
\end{tikzpicture}

  \caption{A tree-cut decomposition of a graph $G$ (left) and its $(a,b,V_1, V_2)$-segregation (right), for some partition $(V_1, V_2)$ of its vertex set. The vertices of $V_1$ and $V_2$ respectively lie in blue and green bags. Newly introduced bags, corresponding to nodes $s_{1},s_{2}$, are empty. The adhesion of $s_1s_2$ is exactly~$E_G(V_1, V_2)$.}
  \label{drugstorefig}
\end{figure}
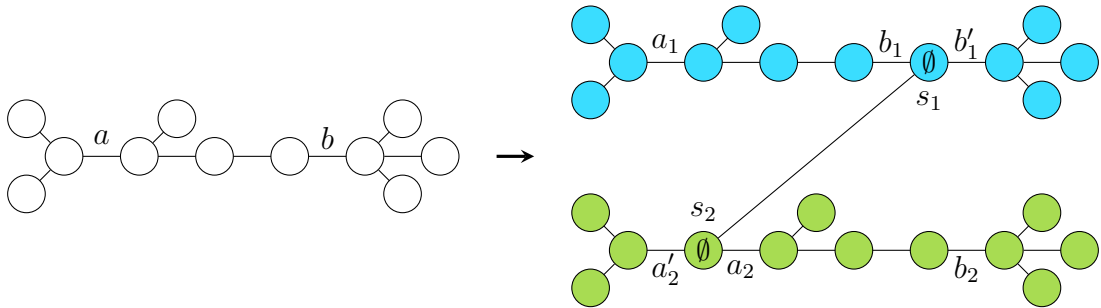

\begin{definition}[segregation of a tree-cut decomposition]\label{possesses}
  Let $(T, \X)$ be a tree-cut decomposition of a graph $G$, let $a,b
  \in E(T)$, and let $(V_1,V_2)$ be a partition of $V(G)$.
We define the \emph{$(a, b, V_1, V_2)$-segregation} of $(T, \X)$ as the pair $(U, \mathcal{Y})$ obtained as follows:
\begin{enumerate}
\item consider a first copy $U_1$ of $T$, subdivide once the link corresponding to $b$, call $s_1$ the subdivision node, and call the two created links $b_1$ and $b_1'$, with the convention that (the copy of) $a$ is closer to $b_1$ in $U_{1}$ (if $a=b$, choose arbitrarily);
\item symmetrically, consider a second copy $U_2$ of $T$, subdivide once the link corresponding to $a$, call $s_{2}$ the subdivision node, and call the two created links $a_2$ and $a_2'$, with the convention that (the copy of) $b$ is closer to $a_2$ in $U_{2}$, or, if $a=b$, coherently as the previous step;
\item in the disjoint union of $U_1$ and $U_2$, add a link joining $s_{1}$ and $s_{2}$: this gives $U$;
\item for every $t \in V(U)$, let $Y_t = 
\begin{cases}
X_t \cap V_1 & \text{if } t\in V(U_1)\setminus \{s_{1}\}\\
 X_t \cap V_2 & \text{if } t\in V(U_2)\setminus \{s_{2}\}, \text{ and }\\
 \emptyset & \text{if } t\in\{s_{1},s_{2}\}.
 \end{cases}$
\end{enumerate}
\end{definition}

An example of a segregation is presented in \autoref{drugstorefig}.
The following remark follows from the definition of a segregation.
\begin{remark}\label{reference}
Any segregation of a tree-cut decomposition of a graph is a tree-cut
decomposition of this graph.
\end{remark}

In order to ensure that the aforementioned improvement steps
eventually lead to a lean tree-cut decomposition, we use the following
notion.
\begin{definition}[fatness]\label{def:fatness}
Let $G$ be a graph on $m$ edges and let
  $(T, \X)$ be a tree-cut decomposition of $G$. For every $i \in
  [m]$, we denote by $T^{\geq i}$ 
the subgraph of $T$ induced by
the links that have an adhesion of size at least~$i$.
  The {\em fatness} of $(T, \X)$ is defined as the $(2m)$-tuple

 \[
   (\alpha_m, -\beta_m,\alpha_{m-1},-\beta_{m-1}, \dots, \alpha_1, -\beta_1),
 \]
where $\alpha_{i}$ is the number of links of
  $T^{\geq i}$ and $\beta_{i}$ is the number of connected components of~$T^{\geq i}$.
We order fatnesses by lexicographic order.
\end{definition}

The following is a slight variant of Menger's Theorem that we use in the next section.
\begin{lemma}\label{testament}
  Let $G$ be a graph, $A,B \subseteq E(G)$, and $k
  \in \N$. 
Then either there is a set of $k$ pairwise edge-disjoint paths of
linking $A$ and~$B$ in $G$, or $G$ has an $(A, B)$-cut of size less than $k$.
\end{lemma}
\begin{proof}
Suppose that there is no set of $k$ pairwise edge-disjoint paths
linking $A$ and~$B$ in~$G$.
We create a new graph $G'$ by subdividing every edge $e \in A \cup B$ into a new vertex~$v_e$.
Define $V_A = \{v_e,\ e \in A\}$ and symmetrically for $V_B$.
Observe then that $G'$ has $k$ pairwise edge-disjoint paths linking
vertex sets $V_A$ and $V_B$ if and only if
$G$ has $k$ pairwise edge-disjoint paths linking edge sets $A$ and
$B$. So, by the edge version of Menger's
Theorem, $G'$ has an edge cut $F'$ of size less than $k$ separating $V_A$ from $V_B$.
Let $F$ be obtained from $F'$ be replacing every edge of the form
$v_e$ for $e \in A\cup B$ (if any) with $e$. It follows that $|F|<k$
and $F$ separates $A$ from $B$ in~$G$.
\end{proof}

\section{The proof of \autoref{pluralism}}
\label{sec:thinghood}

In order to prove \autoref{pluralism}, we use the aforementioned notion of fatness
as a potential that we aim to minimize. 
The strategy we follow is to show that if a tree-cut decomposition of
a graph is not
lean, then it can be modified into a tree-cut decomposition of smaller
fatness, without increasing width (\autoref{dialogues}). As there is no infinite decreasing
sequence of fatnesses of tree-cut
decompositions of a given graph, this process will eventually result
in a lean tree-cut decomposition. Starting from a tree-cut
decomposition of minimum width, we will therefore obtain a lean tree-cut
decomposition of the same width, as desired (\autoref{princeton}).

We first focus on the case where the considered graph is
3-edge-connected, which is the crux of the proof. The reduction from the general case is given
at the end of the section.

\begin{lemma}\label{dialogues}
  Let $w \in \N$ and let $(T, \X)$ be a tree-cut decomposition of width $w$ of a 3-edge connected
  graph $G$.
  If $(T, \X)$ is not lean, then $G$ admits a tree-cut decomposition
  of width at most $w$ that has smaller fatness than $(T, \X)$.
\end{lemma}

\begin{proof}
  Every edge $e$ of $G$ defines two (possibly equal) nodes of $T$,
  those indexing the bags that contain its endpoints. The \emph{$T$-path}
  of $e$ is defined as the path of $T$ linking these vertices.
  For every $e \in E(G)$, we define $d_{a,b}(e)$ as 0 if the $T$-path
  of $e$ shares a link with $aTb$ and $1+d$ otherwise, where $d$ denotes the minimal
  distance between a node of $aTb$ and a node of the $T$-path of~$e$.
  For every set $F \subseteq E(G)$ we set $d_{a,b}(F)= \sum_{e \in F} d_{a,b}(e)$.

  Let us fix a {minimal non-leanness certificate}
  $(k,a,b,A,B)$ of $(T,\X)$. We set $m = |E(G)|$.
  We associate to  $(k,a,b,A,B)$ an $(A,B)$-cut $F$ as follows.
  By \autoref{testament} (the variant of Menger's theorem for edge sets), there is in $G$ a cut of size strictly smaller than $k$ that separates $A$ from $B$.
  Let $F$ be such {an $(A,B)$-cut} of minimal size that, additionally,
  minimizes~$d_{a,b}(F)$.  Notice that none of $A \subseteq F$ and
  $B\subseteq F$ is possible, since $|A|=|B|=k$ and~$|F|<k$. By the
  minimality of the  size of $F$, the graph $G - F$ has exactly
  two connected components.
  We call them $G_{A}$ and $G_{B}$, with the convention that 
  \begin{eqnarray}
    A \subseteq E(G_{A}) \cup F &  \mbox{ and } & B \subseteq E(G_{B})\cup F.\label{corporeal}
  \end{eqnarray}

We denote by $(U, \mathcal{Y})$ the  $(a,b,V(G_A), V(G_B))$-segregation  of $G$. Recall that the tree $U$ is obtained from two copies $U_1$ and $U_2$ of
$T$. For every node $t$ of $T$ and $i\in [2]$, we denote by $t_i$ the
copy of $t$ in $U_i$. Similarly, for every link $e$ of $T$ and $i\in
[2]$, we denote by $e_i$ the copy of $e$ in $U_i$ (except for $a$ and $b$ where the
corresponding subdivided links
have already been named in the definition of a segregation).
Notice that we can unambiguously use $\adh$ without specifying the
tree-cut decomposition it refers to as $E(U)\cap E(T)=\emptyset$.
For every link $e\in E(U) \setminus \{a_2',b_1'\}$, we denote by $\hat{e}$ the corresponding link of $T$, that is, the only link such that $e \in \{\hat{e}_1,\hat{e}_2\}$. For the special cases $e = a_2'$ and $e = b_1'$ we respectively set $\hat{e} = a$ and $\hat{e} = b$.\label{fracturon}

In what follows we prove that 
 $\width(U, \mathcal{Y})\leq  \width(T,\X)$ (\autoref{slem:widthU}) and that the fatness of $(U,
 \mathcal{Y})$ is (strictly) smaller than that of $(T,\X)$ (\autoref{slem:fatfat}).
 The proof is split in a series of sublemmas. 
The end of the proof of each sublemma is marked with a ``$\blacksquare$''. When a sublemma contains a claim, we use the symbol ``{${\diamond}$}'' to mark the end of its proof.
We start with a series of sublemmas related to properties of adhesions.

\begin{sublemma}\label{recurrent}
  For every $e\in E(T)$ and $i \in [2]$,
  \begin{enumerate}[(i)]
  \item $|\adh(e_i)| \leq |\adh(e)|$;\label{doctrines}
  \item $|\adh(a_2')| \leq |\adh(a)|$ and $|\adh(b_1')| \leq |\adh(b)|$;\label{listeners}
  \item if $|\adh(e_i)| = |\adh(e)|$ then $\adh(e_{3-i})
  \subseteq F$;\label{onlooking}
  \item if $|\adh(a_1)| = |\adh(a)|$ then $\adh(a'_{2})
  \subseteq F$ and if $|\adh(b_2)| = |\adh(b)|$ then $\adh(b'_{1})
  \subseteq F$.\label{particuli}
  \end{enumerate}
\end{sublemma}

\begin{proof}[Proof  of \autoref{recurrent}.]
  We assume that $i=1$. The proof for the case $i=2$ is symmetric.
  Before proving the desired inequalities and inclusions, we give some definitions and prove a claim.  
  Let $e\in E(T)$ and let $T_1$ and $T_2$ be the two connected
  components of $T - \{e\}$, named as follows:
  \begin{itemize}
  \item if $e = a = b$, then $T_1$, corresponds (via the isomorphism from $T$ to $U_1$) to the connected component of $U_1 - \{s_1\}$ that is incident with $b_1$, and then $T_2$ to that that is incident with $b_1'$ (recall that in this case, the edge $a=b$ is replaced during the construction of $U_1$ by the two edges $a_1 = b_1$ and $b_1'$, both incident to $s_1$);
  \item otherwise, if $e \notin E(aTb)$ or $e \in \{a,b\}$, we define $T_1$ and $T_2$ so that $aTb$ is disjoint from~$T_1$;
  \item in the remaining case, $e \in E(aTb) \setminus \{a,b\}$ and we choose them so that $a \in V(T_1)$ (and then $b \in V(T_2)$).
  \end{itemize}

  We define $C = \bigcup_{t\in V(T_1)} X_t$ and $D = \bigcup_{t\in V(T_2)} X_t$.

  Also, we~set:
\begin{align}
C_A  = C \cap V(G_A), &\qquad D_A = D \cap V(G_A),  \label{excessive} \\
 C_B  = C \cap V(G_B),&\qquad D_B = D \cap V(G_B). \label{dismissed}
\end{align}
See \autoref{rationale} for an example. 
By the choice of $T_{1}$ and $T_{2}$, every edge in $B$ has an endpoint in $D$.
This, together with  the second statement of~\eqref{corporeal} ensures that:
\begin{eqnarray}
\mbox{Every edge in $B$ either has an endpoint in $D_{B}$ or 
is an edge of $E(D_{A},C_{B})$.}\label{residents}
\end{eqnarray}

Recall that $F$ consists of all edges with  one endpoint in $C_A \cup D_A$ and the other in $C_B\cup D_B$. This means that:
\begin{eqnarray}
F & =& E(C_{A},C_{B})\cup E(C_{A},D_{B})\cup E(D_{A},C_{B})\cup E(D_{A},D_{B}).\label{aesthetes}
\end{eqnarray}
Also, the edges of $\adh(e)$ are those with  one endpoint in $C$ and the other in~$D$.
This implies that:
\begin{eqnarray}
\adh(e) & = & E(C_{A},D_{A})\cup E(C_{A},D_{B})\cup
              E(C_{B},D_{A})\cup E(C_{B},D_{B})\label{inflating}\\
  \text{and}\ \adh(e_1) & = & E(C_{A},D_{A})\cup E(C_{A},C_{B})\cup E(C_{A},D_{B}).\label{sublation}
\end{eqnarray}
We now set:
\begin{eqnarray}
F'& =& (F \setminus E(C_A,C_B)) \cup E(C_B,D_B).\label{generates}
\end{eqnarray}

\begin{claim}\label{cl:abcut}
$F'$ is an $(A,B)$-cut and also an $(E(C_A,D_A), B)$-cut.
\end{claim}
\noindent (The second statement will be useful when proving~\itemref{onlooking}.)

\begin{proof}
Looking for a contradiction, let us assume that there is  a path $P$
linking an edge of $A \cup E(C_A, D_A)$ and an edge of~$B$ in $G - F'$. 
We denote by $e_{A}$ and $e_{B}$ the 
edges of $P$ that are incident to its endpoints, with $e_{A}\in A \cup E(C_A, D_A)$ and $e_{B}\in B$. From~\eqref{residents}, either $e_{B}$ has an endpoint in $D_{B}$ or 
$e_{B}\in E(D_{A},C_{B})$. We first exclude the case where $e_{B}\in E(D_{A},C_{B})$.
Indeed, if this is the case then, using~\eqref{aesthetes}, we obtain that $e_B\in F\setminus E(C_{A},C_{B})\subseteq F'$ a contradiction. We conclude that $e_{B}$ has an endpoint, say $y$, where $y\in D_{B}$.

As $P$ links an edge of $F\cup E(G_{A})$ to an edge of $F\cup E(G_{B})$, $P$ contains at least one edge of $F$.  Since $P$ does not contain edges of $F'$, we deduce that this edge belongs to $E(C_A,C_B)$.
Among all edges of $P$ that belong to $E(C_{A},C_{B})$
let $g$ be the one that is closer to $e_{B}$ in $P$ and let $x_{g}$ be the endpoint of $g$
that belongs to $C_{B}$. By the choice of $g$, we know that the subpath $P'$ of $P$
that is between $x_{g}$ and $y$ is a subgraph of $G_{B}$.
As $x_g\in C_{B}$ and $y\in D_{B}$, we have that $P'$ (and therefore $P$ as well) 
contains an edge $f \in E(C_B,D_B)$. However $E(C_B,D_B) \subseteq F'$, a contradiction. Therefore $F'$ is indeed an $(A,B)$-cut and a $(E(C_A,D_A), B)$-cut. The claim follows.
\renewcommand{\qedsymbol}{$\diamond$}
\end{proof}
\renewcommand{\qedsymbol}{$\square$}

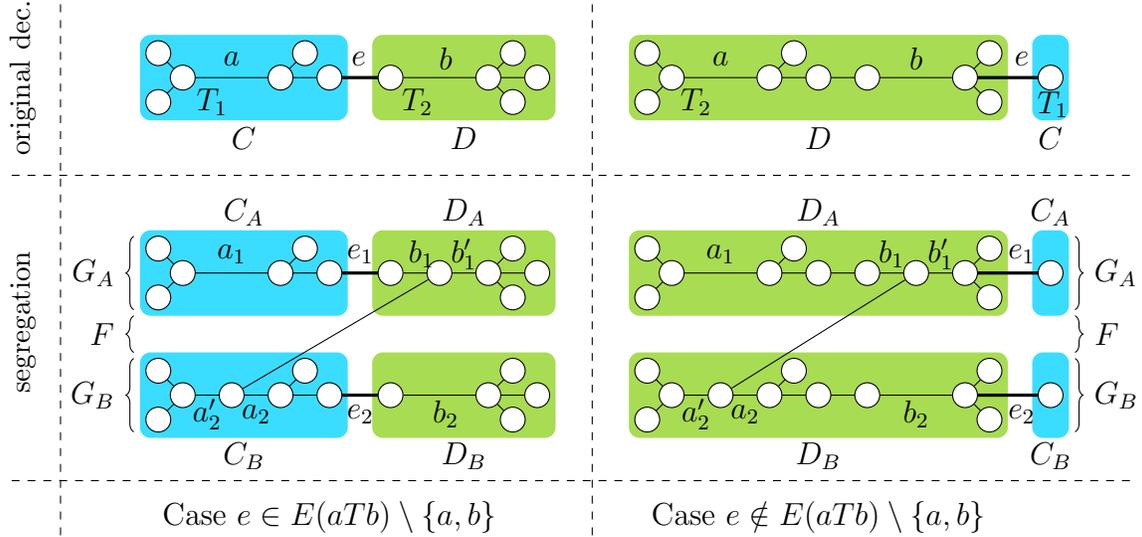
\begin{figure}[H]
  \centering
  \begin{tikzpicture}[scale = 0.58, every node/.style = smallbag]
    \begin{scope}[scale = 0.5] 
  \begin{scope}
    \fill[rounded corners, color = colorD, xshift = 0.5cm] (14.75, 1.75) -- (14.75, -1.75) -- (7.25, -1.75) -- (7.25, 1.75) -- cycle;
    \fill[rounded corners, color = colorC] (-1.75, 1.75) -- (-1.75, -1.75) -- (6.75, -1.75) -- (6.75, 1.75) -- cycle;
    \draw (1.2, -1) node[normal] {$T_{1}$};   
     \draw (2.5, -2.5) node[normal] {$C$};
         \draw (9.6, -1) node[normal] {$T_{2}$};   
    \draw (11.5, -2.5) node[normal] {$D$};
    \draw (0,0) node (b0) {}
    -- ++(4,0) node (b2) {}
    node[normal,midway, yshift = 0.25cm] {$a$}
    -- ++(2,0) node (v0) {};
    \draw[very thick]  (v0) -- ++(2.5,0) node (v1) {} 
    node[normal,midway, yshift = 0.25cm, xshift = -0.1cm] {$e$};
    \draw (v1)-- ++(4,0) node (b1) {}
    node[normal,midway, yshift = 0.25cm] {$b$}
    (b1) -- ++(1,1) node {}
    (b1) -- ++(1,-1) node {}
    (b1) -- ++(2,0) node {};
    \draw
    (b0) -- ++(-1,1) node {}
    (b0) -- ++(-1,-1) node {}
    (b2) -- ++(1,1) node {};
  \end{scope}
  
  \begin{scope}[yshift = -8cm]
    \begin{scope}[every node/.style = smallbag]
      \fill[rounded corners, color = colorD, xshift = 0.5cm] (14.75, 1.75) -- (14.75, -1.75) -- (7.25, -1.75) -- (7.25, 1.75) -- cycle;
      \fill[rounded corners, color = colorC] (-1.75, 1.75) -- (-1.75, -1.75) -- (6.75, -1.75) -- (6.75, 1.75) -- cycle;
      \draw (2.5, 2.5) node[normal] {$C_A$};
      \draw (11.5, 2.5) node[normal] {$D_A$};
      \draw
      (0,0) node (b0) {}
      -- ++(4,0) node (b2) {}
      node[normal,midway, yshift = 0.25cm] {$a_1$}
      -- ++(2,0) node (v0) {};
      \draw[very thick]  (v0) -- ++(2.5,0) node (v1) {} 
      node[normal,midway, yshift = 0.25cm, xshift = -0.08cm] {$e_1$};
      \draw (v1) -- ++(2,0) node (uf) {}
      node[midway,normal, yshift = 0.25cm] {$b_1$}
      -- ++(2,0) node (b1) {}
      node[midway,normal, yshift = 0.25cm] {$b_1'$}
      (b1) -- ++(1,1) node {}
      (b1) -- ++(1,-1) node {}
      (b1) -- ++(2,0) node {};
      \draw
      (b0) -- ++(-1,1) node {}
      (b0) -- ++(-1,-1) node {}
      (b2) -- ++(1,1) node {};
%
    \begin{scope}[yshift = -5cm]
    \fill[rounded corners, color = colorD, xshift = 0.5cm] (14.75, 1.75) -- (14.75, -1.75) -- (7.25, -1.75) -- (7.25, 1.75) -- cycle;
    \fill[rounded corners, color = colorC] (-1.75, 1.75) -- (-1.75, -1.75) -- (6.75, -1.75) -- (6.75, 1.75) -- cycle;
    \draw (2.5, -2.5) node[normal] {$C_B$};
    \draw (11.5, -2.5) node[normal] {$D_B$};
    \draw
    (0,0) node (b0) {}
    -- ++(2,0) node (ue) {}
    node[midway,normal, yshift = -0.25cm] {$a_2'$}
    -- ++(2,0) node (b2) {}
    node[midway,normal, yshift = -0.25cm] {$a_2$}
    -- ++(2,0) node (v0) {};
    \draw[very thick]  (v0) -- ++(2.5,0) node (v1) {} 
    node[normal,midway, yshift = -0.25cm, xshift = -0.08cm] {$e_2$};
    \draw (v1) -- ++(4,0) node (b1) {}
    node[normal,midway, yshift = -0.25cm] {$b_2$}
    (b1) -- ++(1,1) node {}
    (b1) -- ++(1,-1) node {}
    (b1) -- ++(2,0) node {}
    (b0) -- ++(-1,1) node {}
    (b0) -- ++(-1,-1) node {}
    (b2) -- ++(1,1) node {};
    \draw (uf) -- (ue);
  \end{scope}    
\end{scope}

  \end{scope}      
    \end{scope}
\begin{scope}[scale = 0.5, xshift = 20cm] 
  \begin{scope}
    \fill[rounded corners, color = colorC, xshift = 0.5cm] (15.75, 1.75) -- (15.75, -1.75) -- (14.25, -1.75) -- (14.25, 1.75) -- cycle;
    \fill[rounded corners, color = colorD] (-1.75, 1.75) -- (-1.75, -1.75) -- (13.75, -1.75) -- (13.75,1.75) -- cycle;
    \draw (6, -2.5) node[normal] {$D$};
    \draw (1, -1) node[normal] {$T_{2}$};
    \draw (15.5, -2.5) node[normal] {$C$};
    \draw (15.6, -1.1) node[normal] {$T_{1}$};
    \draw (0,0) node (b0) {}
    -- ++(4,0) node (b2) {}
    node[normal,midway, yshift = 0.25cm] {$a$}
    -- ++(2,0) node {}
    -- ++(2,0) node {}
    -- ++(4,0) node (b1) {}
    node[normal,midway, yshift = 0.25cm] {$b$}
    (b1) -- ++(1,1) node {}
    (b1) -- ++(1,-1) node {};
    \draw[very thick] (b1) -- ++(3.5,0) node {} 
    node[normal,midway, yshift = 0.25cm, xshift = 0.1cm] {$e$};
    \draw (b0) -- ++(-1,1) node {}
    (b0) -- ++(-1,-1) node {}
    (b2) -- ++(1,1) node {};
  \end{scope}

  \begin{scope}[yshift = -8cm]
    \begin{scope}[every node/.style = smallbag]
      \fill[rounded corners, color = colorC, xshift = 0.5cm] (15.75, 1.75) -- (15.75, -1.75) -- (14.25, -1.75) -- (14.25, 1.75) -- cycle;
      \fill[rounded corners, color = colorD] (-1.75, 1.75) -- (-1.75, -1.75) -- (13.75, -1.75) -- (13.75, 1.75) -- cycle;
      \draw (6, 2.5) node[normal] {$D_A$};
      \draw (15.5, 2.5) node[normal] {$C_A$};
      \draw
    (0,0) node (b0) {}
    -- ++(4,0) node (b2) {}
    node[normal,midway, yshift = 0.25cm] {$a_1$}
    -- ++(2,0) node {}
    -- ++(2,0) node {}
    -- ++(2,0) node (uf) {}
    node[midway,normal, yshift = 0.25cm] {$b_1$}
    -- ++(2,0) node (b1) {}
    node[midway,normal, yshift = 0.25cm] {$b_1'$}
    (b1) -- ++(1,1) node {}
    (b1) -- ++(1,-1) node {};
    \draw[very thick] (b1) -- ++(3.5,0) node {} 
    node[normal,midway, yshift = 0.25cm, xshift = 0.1cm] {$e_1$};
    \draw (b0) -- ++(-1,1) node {}
    (b0) -- ++(-1,-1) node {}
    (b2) -- ++(1,1) node {};
    \begin{scope}[yshift = -5cm]
      \draw [decorate,decoration={brace,amplitude=3pt}] (16.5, 3.25) -- (16.5, 1.75) node[normal, midway, anchor = west, xshift = 0.25cm] {$F$};
      \draw [decorate,decoration={brace,amplitude=3pt}] (16.5, 6.5) -- (16.5, 3.5) node[normal, midway, anchor = west, xshift = 0.25cm] {$G_A$};
      \draw [decorate,decoration={brace,amplitude=3pt}] (16.5, 1.5) -- (16.5, -1.5) node[normal, midway, anchor = west, xshift = 0.25cm] {$G_B$};
      \draw [decorate,decoration={brace,amplitude=3pt}, xshift = 2cm] (-24, 1.75) -- (-24, 3.25) node[normal, midway, anchor = east, xshift = -0.25cm] {$F$};
      \draw [decorate,decoration={brace,amplitude=3pt}, xshift = 2cm] (-24, 3.5) -- (-24, 6.5) node[normal, midway, anchor = east, xshift = -0.25cm] {$G_A$};
      \draw [decorate,decoration={brace,amplitude=3pt}, xshift = 2cm] (-24, -1.5) -- (-24, 1.5) node[normal, midway, anchor = east, xshift = -0.25cm] {$G_B$};
    \fill[rounded corners, color = colorC, xshift = 0.5cm] (15.75, 1.75) -- (15.75, -1.75) -- (14.25, -1.75) -- (14.25, 1.75) -- cycle;
      \fill[rounded corners, color = colorD] (-1.75, 1.75) -- (-1.75, -1.75) -- (13.75, -1.75) -- (13.75,1.75) -- cycle;
      \draw (6, -2.5) node[normal] {$D_B$};
      \draw (15.5, -2.5) node[normal] {$C_B$};
    \draw
    (0,0) node (b0) {}
    -- ++(2,0) node (ue) {}
    node[midway,normal, yshift = -0.25cm] {$a_2'$}
    -- ++(2,0) node (b2) {}
    node[midway,normal, yshift = -0.25cm] {$a_2$}
    -- ++(2,0) node {}
    -- ++(2,0) node {}
    -- ++(4,0) node (b1) {}
    node[normal,midway, yshift = -0.25cm] {$b_2$}
    (b1) -- ++(1,1) node {}
    (b1) -- ++(1,-1) node {};
    \draw[very thick] (b1) -- ++(3.5, 0) node {} 
    node[normal,midway, yshift = -0.25cm, xshift = 0.1cm] {$e_2$};
    \draw (b0) -- ++(-1,1) node {}
    (b0) -- ++(-1,-1) node {}
    (b2) -- ++(1,1) node {};
    \draw (uf) -- (ue);
    \end{scope}
    \end{scope}
  \end{scope}
\end{scope}
\draw[dashed] (8.375, 1.5) -- ++(0,-11);
\draw[dashed] (-2.5, 1.5) -- ++(0,-11);
\draw[dashed] (-3.5, -2) -- ++(23,0);
\draw[dashed] (-3.5, -8.25) -- ++(23,0);
\draw (3, -9) node[normal] {Case $e \in E(aTb)\setminus
  \{a,b\}$};
\draw (13, -9) node[normal] {Case $e \notin E(aTb)\setminus
  \{a,b\}$};
\draw (-3, 0) node[normal, anchor = south, rotate = 90, text width = 4cm, align = center] {{\small original dec.}};
\draw (-3, -5) node[normal, anchor = south, rotate = 90] {{\small segregation}};
\end{tikzpicture}
 \caption{The sets $C$, $D$, $C_A$, $D_A$, $C_B$, and $D_B$ in the
    original decomposition and in the segregation considered in the proof of \autoref{recurrent}, depending on the position of $e$. Note that $e$ could also be an edge of $E(T) \setminus E(aTb)$ such that its closest vertex on $aTb$ is an internal vertex.}
  \label{rationale}
\end{figure}

Now we prove \itemref{doctrines}--\itemref{particuli}.
\medskip

\noindent \textit{Proof of \itemref{doctrines}.} 
Assume for contradiction that $|\adh(e_1)| > |\adh(e)|$. 
Comparing~\eqref{inflating} with~\eqref{sublation}, we get $|E(C_A, C_B)| > |E(C_B, D_A)| + |E(C_B, D_B)|$ which implies:
\begin{eqnarray}
|E(C_A, C_B)| & >  & |E(C_B, D_B)|.
\label{magnitude}
\end{eqnarray}
In particular $E(C_A, C_B)$ is non-empty. From~\eqref{aesthetes}, $E(C_A, C_B)
\subseteq F$.
Using~\eqref{magnitude} and the definition of $F'$ in~\eqref{generates}, 
 we deduce that $|F'| < |F|$, a contradiction to the minimality of~$|F|$.
This proves~\itemref{doctrines}.
\medskip

\noindent \textit{Proof of \itemref{listeners}.}
The proof is identical to the proof of \itemref{doctrines}, using $e=a$, $i=2$, $e_2=a_2'$ (resp.\ $e=b$, $i=1$, $e_{1}=b_1'$)  to get the first (resp.\ second) inequality.
\medskip

\noindent \textit{Proof of \itemref{onlooking}.}
Let us assume that $|\adh(e_1)| = |\adh(e)|$ (the case where
$|\adh(e_2)| = |\adh(e)|$ is symmetric).
As we assume $|\adh(e_1)| = |\adh(e)|$, using~\eqref{inflating} and~\eqref{sublation}, we get
\begin{align}
  |E(C_A,C_B)| &\geq |E(C_B,D_B)|.\label{unbinding} 
\end{align}
Considering $F'$ as defined in~\eqref{generates}, we deduce from~\eqref{unbinding} that~$|F'| \leq |F|$, which we will use later.
Towards a contradiction with \itemref{onlooking}, let us assume that $\adh(e_{2})
  \nsubseteq F$ or, equivalently, that $|\adh(e_{2}) \setminus F|>0$. 
We consider two different cases.
\smallskip

\noindent \textit{First case:} $e \not \in E(aTb)$ or $e \in
\{a,b\}$. We notice the following equality:

\begin{eqnarray}
\adh(e_{2}) & = & E(C_A,C_B) \cup E(C_B,D_A) \cup E(C_B, D_B).\label{falseness}
\end{eqnarray}

Eq.\ \eqref{falseness} together with~\eqref{aesthetes}, implies that
$\adh(e_{2}) \setminus F = E(C_B,D_B)$ and we deduce
$|E(C_B,D_B)|>0$.
With \eqref{unbinding} this implies that $E(C_A,C_B)$ is non-empty. 
Also, from \autoref{cl:abcut} we have that $F'$ is an $(A,B)$-cut and
we proved above that it is not larger than~$F$.  Notice that the
$T$-path of any edge of $E(C_B,D_B)$ contains $e$.
On the other hand,
\begin{itemize}
\item when $e \in \{a,b\}$, no $T$-path of an edge of $E(C_A, C_B)$
  does contain~$e$, and
\item when $e \notin E(aTb)$, no $T$-path of an edge of
$E(C_A, C_B)$ does contain the endpoint $t_e$ of $e$ that is the closest to
a node of~$aTb$.
\end{itemize}
Therefore, for every $f \in E(C_A, C_B)$ and $f' \in E(C_B,D_B)$,
\begin{itemize}
\item either $e \in \{a,b\}$, then $d_{a,b}(f') = 0$ (because the $T$-path of
  $f'$ contains an edge of $aTb$, which is $e$) and $d_{a,b}(f) > 0$
  (for the opposite reason);
\item or $e \notin E(aTb)$, then $d_{a,b}(f') \leq \dist_T(V(aTb),
  t_e)$ (because the $T$-path of $f'$ contains $t_e$) and $d_{a,b}(f) \geq \dist_T(V(aTb),
  t_e) +1$ (for the opposite reason, and by definition of $t_e$).
\end{itemize}
In both cases we have $d_{a,b}(f') < d_{a,b}(f)$.
The fact that  $E(C_A,C_B)$ is non-empty, together
with~\eqref{unbinding}, imply that  $d_{a,b}(E(C_A,C_B)) >
d_{a,b}(E(C_B,D_B)) \geq 0$, hence $d_{a,b}(F') < d_{a,b}(F)$. This
contradicts the choice of $F$, thus this case is not possible.
\smallskip

\noindent \textit{Second case:} $e \in E(aTb) \setminus \{a,b\}$. 
Recall (\autoref{cl:abcut}) that $F'$ is a $(E(C_A,D_A),
B)$-cut. Notice that because of~\eqref{aesthetes}
and~\eqref{generates}, it follows that $\adh(e) \setminus F' =
E(C_A,D_A)$. We deduce that $F'$, in fact, a $(\adh(e), B)$-cut.

We choose a subset $F''$ of $\adh(e)$ such that $|F'' |=k$.
(Because $e \in aTb$ and of the definition of $a,b$,  $|\adh(e)| \geq
k$ and such a subset always exists.)
We claim that the quintuple $(k, e, b, F'', B)$ satisfies conditions
\itemref{distanced} and \itemref{commanded} of~\autoref{forbidden}
(non-leanness certificate).
Since $eTb$ is a subpath of $aTb$, we have $|\adh(e')| \geq k$ for
every $e' \in E(eTb)$ and thus Condition \itemref{commanded}
holds. For Condition  \itemref{distanced} observe that $F'$ separates
$\adh(e)$ from $B$ and $|F'| \leq |F| <k$, therefore there are no $k$
edge-disjoint paths linking $\adh(e)$ to $B$.

Besides Conditions \itemref{distanced} and \itemref{commanded}, $e
\in E(aTb) \setminus \{a,b\}$, and therefore $eTb$ is shorter than $aTb$. This
contradicts the minimality of the distance between $a$ and $b$ that we
assumed. Therefore, this case is not possible either and we have in
both cases that $\adh(e_2)\subseteq F$.
\medskip

\noindent \textit{Proof of \itemref{particuli}.} The proof follows
the very same steps as the proof of \itemref{onlooking} (first case)
using $e=a$, $i=1$, $e_2=a_2'$ (resp.\ $e=b$, $i=2$, $e_1=b_1'$).
\renewcommand{\qedsymbol}{$\blacksquare$}
\end{proof}
\renewcommand{\qedsymbol}{$\square$}

We can now complete the first goal of this proof.
\begin{sublemma}\label{slem:widthU}
  $\width(U, \mathcal{Y})\leq  \width(T,\X)$
\end{sublemma}
\begin{proof}
From~\autoref*{recurrent}.\itemref{doctrines} and \autoref*{recurrent}.\itemref{listeners}
we obtain that 
\[
  \max_{g\in E(U)} |\adh(g)| \leq \max_{g\in E(T)} |\adh(g)|.
\]

By definition of a segregation, for every $i \in [2]$ and $t \in
V(T)$, we have $Y_{t_i} \subseteq X_t$ and $\deg_U(t_i) = \deg_T(t)$,
hence $|X_t| + \deg_T(t) \geq |Y_{t_i}| + \deg_U(t_i)$.
For $i\in [2]$ we also have $Y_{s_i} = \emptyset$, and $\deg_U(s_i) = 3$.
As $G$ is 3-edge-connected and $\adh(a)$ is not empty,
$|\adh(a)| \geq 3$ so in particular, $|\adh(a)| \geq |Y_{s_i}| +
\deg_U(s_i)$.
Using the simplified definition of width for tree-cut decompositions
of 3-edge-connected graphs \eqref{eq:tcwdeg}, we conclude that
$\width(U, \mathcal{Y}) \leq \width(T, \X)$.
\renewcommand{\qedsymbol}{$\blacksquare$}
\end{proof}
\renewcommand{\qedsymbol}{$\square$}

In the rest of the proof we focus on the second goal, i.e.\ showing
that the fatness of $(U, \mathcal{Y})$ is smaller than that of $(T,
\X)$. Let
\[(\alpha_m, -\beta_m, \alpha_{m-1}, -\beta_{m-1}, \dots, \alpha_1,
-\beta_1)\] be the fatness of $(T, \X)$ and let \[(\alpha'_m,
-\beta'_m, \alpha'_{m-1}, -\beta'_{m-1}, \dots, \alpha'_1, -\beta'_1)\]
be that of $(U, \mathcal{Y})$, as defined in \autoref{def:fatness}
(recall that $m=|E(G)|$).

\begin{sublemma}
\label{withstood}
  For every $e \in E(T)$,
  \begin{itemize}
  \item either $|\adh(e)| > |\adh(e_1)|$ and $|\adh(e)| >
    |\adh(e_2)|$;
  \item or there is some     $i \in [2]$ such that $|\adh(e)| = |\adh(e_i)|$ and $\adh(e_{3-i}) \subseteq F$.
  \end{itemize}
\end{sublemma}

\begin{proof}[Proof  of \autoref{withstood}] This sublemma is a direct corollary of~\autoref*{recurrent}.\itemref{doctrines}
and~\autoref*{recurrent}.\itemref{onlooking}.\renewcommand{\qedsymbol}{$\blacksquare$}
\end{proof}
\renewcommand{\qedsymbol}{$\square$}

\begin{sublemma}\label{vegetable}
 $A\subseteq  \adh(a_1)\cup F$ and  $B\subseteq  \adh(b_2)\cup F$.
  \end{sublemma} 

\begin{proof}[Proof  of \autoref{vegetable}.]
We only prove the first statement as the proof of the second one is symmetric.
We define $C_{A}$, $C_{B}$, $D_{A}$, and $D_{B}$ as in~\eqref{excessive} and~\eqref{dismissed}
in the proof of~\autoref{recurrent} for the case where $e=a$.
Under this setting,~\eqref{aesthetes},~\eqref{inflating}, and~\eqref{sublation} are still valid; we restate them below for clarity.
\begin{eqnarray}
F & =& E(C_{A},C_{B})\cup E(C_{A},D_{B})\cup E(D_{A},C_{B})\cup E(D_{A},D_{B}),\label{mightiest}\\
\adh(a_1) & = & E(C_{A},D_{A})\cup E(C_{A},C_{B})\cup E(C_{A},D_{B}),\ \text{and}\label{affluence}\\
\adh(a) & = & E(C_{A},D_{A})\cup E(C_{A},D_{B})\cup  E(C_{B},D_{A})\cup E(C_{B},D_{B}).\label{fruitless}
\end{eqnarray}
Recall that $A \subseteq E(G_{A}) \cup F $
(see~\eqref{corporeal}). Also, from~\eqref{dismissed}, we obtain that
$E(C_{B},D_{B})\subseteq E(G_{B})$. These two relations imply that
$A\cap E(C_{B},D_{B})=\emptyset$. Combining this last relation
with~\eqref{fruitless}, we get:
\begin{eqnarray}
A & \subseteq & E(C_{A},D_{A})\cup E(C_{A},D_{B})\cup  E(C_{B},D_{A}).\label{graveyard}
\end{eqnarray}
As each of the terms of the right side of~\eqref{graveyard}
appears on the right side of either~\eqref{mightiest} or~\eqref{affluence}, we conclude that $A\subseteq F\cup \adh(a_1)$ as required.
\renewcommand{\qedsymbol}{$\blacksquare$}
\end{proof}
\renewcommand{\qedsymbol}{$\square$}

Given an integer $p$, we say that a link $e\in E(T)$ is {\em $p$-excessive} if 
\begin{eqnarray}
|\adh(e)|\geq p, & |\adh(e)|> |\adh(e_1)|, & 
\mbox{and } |\adh(e)|>|\adh(e_2)|.\label{ingenuity}
\end{eqnarray}

\begin{sublemma}\label{committal}
  Let $l$ be an integer such that $l\geq k$  and none of the  links  of 
  $T$ of adhesion more than $l$ is $k$-excessive. 
  Then 
  \begin{enumerate}[(i)]
  \item \label{item:alphaineq} $\alpha'_l \leq \alpha_l$ and $\beta'_l
    \geq \beta_l$, and
  \item \label{item:alphaeq} for every $j \in [l+1, m]$, $\alpha'_j = \alpha_j$  and $\beta'_j \geq \beta_j$.
  \end{enumerate}
\end{sublemma}

\begin{proof}[Proof  of \autoref{committal}.]  Let $j \in [l, m]$.
  Recall that we denote by $U^{\geq j}$ the subgraph of $U$  induced by links that have an adhesion of size at least~$j$.
  We need first the following claim:
  
  \begin{claim}\label{cl:fiuj}
    If  $f$ is a link of $T$ such that   $f_i$ belongs to $U^{\geq j}$ for some  $i\in[2]$, then $|\adh(f)|=|\adh(f_i)|$ and  $\adh(f_{3-i})\subseteq F$.
  \end{claim}
  
  \begin{proof}
    We first prove that it is not possible that $|\adh(f)| > |\adh(f_1)|$ and
    $|\adh(f)|>|\adh(f_2)|$. Towards a contradiction, let us suppose that it
    holds. If $|\adh(f)|> l$, then $f$ is a $k$-excessive
    link of adhesion greater than $l$, a contradiction to the hypothesis of the lemma.
    If $|\adh(f)|\leq l$, then for every
    $i\in[2]$ we have $|\adh(f_{i})|< l\leq j$, hence $f_i\not\in E(U^{\geq j})$, a contradiction.
    By \autoref{withstood}, there is some $i' \in [2]$ such that
    $|\adh(f)| = |\adh(f_{i'})|$ and $\adh(f_{3-i'}) \subseteq F$. As
    $|\adh(f_{3-i'})|\leq |F|< k\leq j\leq |\adh(f_i)|$, we have that
    $i=i'$, and therefore $|\adh(f)| = |\adh(f_i)|$ and
    $\adh(f_{3-i})\subseteq F$, as desired.
    \renewcommand{\qedsymbol}{$\diamond$}
  \end{proof}
  \renewcommand{\qedsymbol}{$\square$}
  
  Let $f \in E(T)$. From the above claim we have the following:
  \begin{eqnarray}
    \mbox{If  $f_i\in E(U^{\geq j})$ for some  $i\in [2]$, then $f\in E(T^{\geq j})$ and  $f_{3-i}\not\in E(U^{\geq j})$.}
    \label{unhappily}
  \end{eqnarray}

  We  next claim that if $f\in\{a,b\}$ and $f_i$ belongs to $U^{\geq
    j}$ for some $i\in [2]$, then $i=1$ in case $f=a$ and $i=2$ in
  case $f=b$. We present the proof of this claim for the case where
  $f=a$ (the case $f=b$ is symmetric). Assume to the contrary that
  $i=2$.  Then, from the above claim, $\adh(a_{1})\subseteq F$. Recall
  that $A \subseteq \adh(a_1)\cup F$, according
  to~\autoref{vegetable}. We conclude that $A\subseteq  F$, a
  contradiction as $|A|=k$ and $|F|<k$. Thus, the claim
  holds.\smallskip

  By \autoref{cl:fiuj}, if $a_1\in E(U^{\geq j})$ then $|\adh(a_1)| =
  |\adh(a)|$ and  $|\adh(a_{2})|\leq  |F|<k\leq j$. Therefore $a\in
  E(T^{\geq j})$ and  $a_{2}\not\in E(U^{\geq j})$. Moreover, the fact
  that $|\adh(a_1)| = |\adh(a)|$ together with the  first statement of
  \autoref*{recurrent}.\itemref{particuli} implies that $\adh(a'_{2})
  \subseteq F$. This implies that $|\adh(a_2')| \leq |F| < k\leq j$,
  therefore $a_{2}'\not\in E(U^{\geq j})$. We resume these
  observations, along with the symmetric observations for the case
  where $b_2\in E(U^{\geq j})$, to the following statements:
  \begin{align}
    & \text{If}\ a_1\in E(U^{\geq j}),\ \text{then}\ a\in E(T^{\geq j}),\ a_{2}\not\in E(U^{\geq j}),\ \text{and}\ a_{2}'\not\in E(U^{\geq j}).
      \label{character}
    \\
    &\text{If}\ b_2\in E(U^{\geq j}),\ \text{then}\ b\in E(T^{\geq j}),\ b_{1}\not\in E(U^{\geq j}), \text{and}\ b_{1}'\not\in E(U^{\geq j}).
      \label{defensive}
  \end{align}

  Let $j \in [l, m]$ and let us define now the function
  $\varphi \colon E(U^{\geq j}) \to E(T^{\geq j})$ so that
  $\varphi(e) = \hat{e}$ for every $e\in E\left (U^{\geq j}\right
  )$. (Recall that $\hat{e}$ is the edge of $T$ from which $e$ has
  been copied, see the paragraph following~\eqref{corporeal} for the
  definition.) According to~\eqref{unhappily},~\eqref{character},
  and~\eqref{defensive}, the function
  $\varphi$  is injective. Hence  $\left |E \left (U^{\geq j} \right ) \right| \leq
  \left |E \left (T^{\geq j} \right ) \right |$, i.e., $\alpha'_j \leq \alpha_j$. This
  proves the first half of~\eqref{item:alphaineq}.

  When $j\in [l+1, m]$, the function $\varphi$ is even
  surjective: by definition of $l$, every link $f \in E(T^{\ge j})$
  satisfies $|\adh(f)| = |\adh(f_i)|$ for some $i \in [2]$, therefore
  $f_i \in U^{\geq j}$ is the preimage of $f$ by~$\varphi$.  As a
  consequence, for every $j \in [l+1, m]$, we have $|E(U^{\geq
      j})| = |E( T^{\ge j} )|$, that is, $\alpha'_j = \alpha_j$ and the first
  part of~\eqref{item:alphaeq} holds.\smallskip

  We now deal with the second parts of~\eqref{item:alphaineq}
  and~\eqref{item:alphaeq}. Let $j \in [l, m]$; we will prove
  that $\beta_j' \geq \beta_j$.  Recall that $s_1s_2$ is
  the link joining the nodes $s_{1}$ and $s_{2}$ in $U$ and
  $\adh(s_1s_2)=F$. As $|F|<k\leq l \leq j$, we have that $s_1s_2\not\in
  E(U^{\geq j})$. This means that none of the connected components of $U^{\geq
    j}$ contains $s_1s_2$.  Let $Q$ be a connected component of $U^{\geq
    j}$. Then, from~\eqref{unhappily}, for every $f\in E(Q)$ it holds that
  $\hat{f}\in E(T^{\geq j})$.  Therefore, if $Q$ is a connected
  component of $U^{\geq j}$, then the subgraph $T_{Q}$ of $T$ with link
  set $\{\hat{f}, f \in E(Q)\}$ is a (connected) subtree of $T^{\geq
    j}$. Let $\psi$ be the function that maps every connected component
  $C$ of $U^{\geq j}$ to the connected component of $T^{\geq j}$ that
  contains the subgraph $T_{Q}$, defined as above.  Recall that by
  definition of $l$, if
  $f$ is a link of $T^{\geq j}$ then $f$ is not $k$-excessive and, thus,
  $f_i$ is a link in $U^{\geq j}$, for some $i \in [2]$. Therefore, the
  connected component of $T^{\geq j}$ containing $f$ is the image by
  $\psi$ of the connected component of $U^{\geq j}$ containing
  $f_i$. This proves that $\psi$ is surjective. The forest $U^{\geq j}$
  then {has} at least as many connected components as $T^{\geq j}$ or,
  in other words, $\beta'_j \geq \beta_j$. This proves the second
  part of~\eqref{item:alphaeq} and concludes the proof.
  \renewcommand{\qedsymbol}{$\blacksquare$}
\end{proof} \renewcommand{\qedsymbol}{$\square$}

\begin{sublemma}
  \label{reexamine}
  If $T$ has a $k$-excessive link, then there is an integer $l \geq k$
  such that 
  \begin{enumerate}[(i)]
  \item $\alpha'_l < \alpha_l$ and
  \item\label{item:alphaeq2} for every $j \in [l+1, m]$, $\alpha'_j = \alpha_j$ and $\beta_j' \geq \beta_j$.
  \end{enumerate}  
\end{sublemma}

\begin{proof}[Proof of \autoref{reexamine}]
  Let $g$ be an $k$-excessive link of maximum adhesion and 
  let $l = |\adh(g)|$. By definition of $k$-excessive (see~\eqref{ingenuity}),  we have~$l\geq k$.
  By the choice of $g$, the integer $l$ satisfies the requirements of \autoref{committal}. 
  Item~\eqref{item:alphaeq2} then directly follows.
  Let us consider the same function $\varphi$ as in the proof of
  \autoref{committal} (i.e., we set $\varphi(e) = \hat{e}$).  
  By definition of $g$, we have $g\in E(T^{\ge l})$ whereas $g_1,g_2 \not \in
  E(U^{\ge l})$. Therefore $g$ has no preimage in $E(U^{\ge l})$ by $\varphi$: this
  function is not surjective.
  Thus, $|E( U^{\ge l} )| < |E( T^{\ge l} )|$, or, equivalently, $\alpha'_l < \alpha_l$.
\renewcommand{\qedsymbol}{$\blacksquare$}
\end{proof}
\renewcommand{\qedsymbol}{$\square$}

\begin{sublemma}\label{classical}
  If $a$ and $b$ are distinct and not incident in $T$, then $aTb$ has at least one $k$-excessive link.
\end{sublemma}

\begin{proof}[Proof  of \autoref{classical}]
  Towards a contradiction, let us assume the opposite statement: for
  every link $e \in E(aTb)$, (at least) one of the following holds:
  $|\adh(e)| \leq |\adh(e_1)|$ or $|\adh(e)|\leq|\adh(e_2)|$ (the case
  where $|\adh(e)|<k$ is excluded because $e \in E(aTb)$ and Condition
  \itemref{commanded} holds).  Our aim is to find a {non-leanness
    certificate} for $(T,\X)$ that contradicts the minimality of
  $(k,a,b,A,B)$.

  Let $e\in E(aTb)$ and $i\in [2]$ such that $|\adh(e)| \leq
  |\adh(e_i)|$. By \autoref*{recurrent}.\itemref{doctrines}-\itemref{onlooking}, we in
  fact have $|\adh(e)| = |\adh(e_i)|$ and $\adh(e_{3-i}) \subseteq F$.
  In particular $|\adh(e_{3-i})| < |\adh(e)|$, because $|\adh(e)|
  \geq k$, as noted above, while $|F| < k$. We deduce:
   
  \begin{equation}\label{eqn:disjunc}
    \forall e \in E(aTb),\ \exists i\in [2],\
    \left \{
      \begin{array}{rcl}
        |\adh(e_i)| &=& |\adh(e)|\ \text{and}\\
        |\adh(e_{3-i})| &<& |\adh(e)|.
      \end{array}
    \right .
  \end{equation}

  For the case where $e=a$ in~\eqref{eqn:disjunc}, we claim that $i =
  1$, i.e.
  \begin{equation}
    \label{eq:a-is-blue}
    |\adh(a_{1})| = |\adh(a)|\quad \text{and}\quad |\adh(a_{2})|<|\adh(a)|.
  \end{equation}

  If this claim was not correct, then we would have
  $|\adh(a_{2})|=|\adh(a)|$ and by applying
  \autoref*{recurrent}.\itemref{onlooking} for $e=a$ we would get
  $\adh(a_{1})\subseteq F$. Together with $A\subseteq \adh(a_1)\cup F$
  (from \autoref{vegetable}), this would imply $A\subseteq F$. Hence
  $k=|A|\leq |F|<k$, a contradiction.

  By replacing $a,A,a_1$ by $b,B,b_{2}$ in the argument
  above, we can similarly show
  \begin{equation}
    \label{eq:b-is-red}
    |\adh(b_{1})| < |\adh(b)|\quad \text{and}\quad |\adh(b_{2})| = |\adh(b)|.
  \end{equation}
  
  \begin{claim}\label{precursor}
    There
    are two incident links $e,f \in E(aTb)$ such that  
    \begin{enumerate}
    \item $\{e,f\}
      \neq \{a,b\}$, 
    \item $e\in E(aTf)$, and
    \item $|\adh(e)| = |\adh(e_1)|$ and $|\adh(f)| = |\adh(f_{2})|$.
    \end{enumerate}
  \end{claim}
  
  \begin{proof}
    Let us color in blue every edge $e$ of $aTb$ such that $|\adh(e)| =
    |\adh(e_1)|$ and in red every edge such that $|\adh(e)| =
    |\adh(e_2)|$. By the virtue of~\eqref{eqn:disjunc}, every edge
    receives exactly one color. By~\eqref{eq:a-is-blue}
    and~\eqref{eq:b-is-red}, $a$ is colored blue and $b$ is colored
    red. Let $f$ be the first red edge met when following $aTb$
    from~$a$ and let $e$ be the edge met just before. This choice
    ensures the two last desired properties. We assumed that $a$ and $b$ are not incident, so $\{e,f\}
    \neq \{a,b\}$.%
    \renewcommand{\qedsymbol}{$\diamond$}
  \end{proof}
  \renewcommand{\qedsymbol}{$\square$}

\begin{claim}\label{cl:sep}
Let $e,f\in E(T)$ be links satisfying the conditions of \autoref{precursor}. Then  $F$ separates $\adh(e)$ from $\adh(f)$. 
\end{claim}

\begin{proof}
The third condition of \autoref{precursor} along with \autoref*{recurrent}.\itemref{onlooking}, implies that  
 \begin{eqnarray}
\adh(e_{2})
\subseteq F & \mbox{and} & \adh(f_1) \subseteq F.\label{dialectic}
 \end{eqnarray} 

\begin{figure}[h]
  \centering
\begin{tikzpicture}[scale = 0.75, scale = 0.48, every node/.style =
  smallbag]
  \begin{scope}
    \fill[rounded corners, color = colorD, xshift = 1cm] (14.75, 1.75) -- (14.75, -1.75) -- (7.25, -1.75) -- (7.25, 1.75) -- cycle;
    \fill[rounded corners, color = colorC] (-1.75, 1.75) -- (-1.75,
    -1.75) -- (4.75, -1.75) -- (4.75, 1.75) -- cycle;
    \fill[rounded corners, color = colorM] (5.75, 2.25) rectangle (7.25, -1.75);
    \draw (1.75, -2.5) node[normal] {$C$};
    \draw (6.5, -2.5) node[normal] {$M$};
    \draw (12, -2.5) node[normal] {$D$};
    \draw (0,0) node (b0) {}
    -- ++(4,0) node (b2) {}
    node[normal,midway, yshift = 0.25cm] {$a$}
    -- ++(2.5,0) node (v0) {}
    node[normal,midway, yshift = 0.25cm] {$e$};
    \draw (v0) -- ++(2.5,0) node (v1) {} 
    node[normal,midway, yshift = 0.25cm, xshift = -0.1cm] {$f$};
    \draw (v1)-- ++(4,0) node (b1) {}
    node[normal,midway, yshift = 0.25cm] {$b$}
    (b1) -- ++(1,1) node {}
    (b1) -- ++(1,-1) node {}
    (b1) -- ++(2,0) node {};
    \draw
    (b0) -- ++(-1,1) node {}
    (b0) -- ++(-1,-1) node {}
    (v0) -- ++(0,1.5) node {};
  \end{scope}
    \begin{scope}[yshift = -9cm]
    \begin{scope}[every node/.style = smallbag]
    \fill[rounded corners, color = colorD, xshift = 1cm] (14.75, 1.75) -- (14.75, -1.75) -- (7.25, -1.75) -- (7.25, 1.75) -- cycle;
    \fill[rounded corners, color = colorC] (-1.75, 1.75) -- (-1.75,
    -1.75) -- (4.75, -1.75) -- (4.75, 1.75) -- cycle;
    \fill[rounded corners, color = colorM] (5.75, 2.25) rectangle (7.25, -1.75);
      \draw (1.75, 2.5) node[normal] {$C_A$};
      \draw (12, 2.5) node[normal] {$D_A$};
      \draw (6.5, 3) node[normal] {$M_{A}$};
      \draw
      (0,0) node (b0) {}
      -- ++(4,0) node (b2) {}
      node[normal,midway, yshift = 0.25cm] {$a_1$}
      -- ++(2.5,0) node (v0) {}
      node[normal,midway, yshift = 0.25cm] {$e_1$};
      \draw (v0) -- ++(2.5,0) node (v1) {} 
      node[normal,midway, yshift = 0.25cm, xshift = -0.08cm] {$f_1$};
      \draw (v1) -- ++(2,0) node (uf) {}
      node[midway,normal, yshift = 0.25cm] {$b_1$}
      -- ++(2,0) node (b1) {}
      node[midway,normal, yshift = 0.25cm] {$b_1'$}
      (b1) -- ++(1,1) node {}
      (b1) -- ++(1,-1) node {}
      (b1) -- ++(2,0) node {};
      \draw
      (b0) -- ++(-1,1) node {}
      (b0) -- ++(-1,-1) node {}
      (v0) -- ++(0,1.5) node {};
%
      \begin{scope}[yshift = -5cm]
        \draw [decorate,decoration={brace,amplitude=3pt}] (-2, 1.75) -- (-2, 3.25) node[normal, midway, anchor = east, xshift = -0.25cm] {$F$};
    \fill[rounded corners, color = colorD, xshift = 1cm] (14.75, 1.75) -- (14.75, -1.75) -- (7.25, -1.75) -- (7.25, 1.75) -- cycle;
    \fill[rounded corners, color = colorC] (-1.75, 1.75) -- (-1.75,
    -1.75) -- (4.75, -1.75) -- (4.75, 1.75) -- cycle;
    \fill[rounded corners, color = colorM] (5.75, 1.75) rectangle (7.25, -2.25);
    \draw (1.75, -2.5) node[normal] {$C_B$};
    \draw (12, -2.5) node[normal] {$D_B$};
      \draw (6.5, -3) node[normal] {$M_{B}$};
    \draw
    (0,0) node (b0) {}
    -- ++(2,0) node (ue) {}
    node[midway,normal, yshift = -0.25cm] {$a_2'$}
    -- ++(2,0) node (b2) {}
    node[midway,normal, yshift = -0.25cm] {$a_2$}
    -- ++(2.5,0) node (v0) {}
    node[midway,normal, yshift = -0.25cm] {$e_2$};
    \draw (v0) -- ++(2.5,0) node (v1) {} 
    node[normal,midway, yshift = -0.25cm, xshift = -0.08cm] {$f_2$};
    \draw (v1) -- ++(4,0) node (b1) {}
    node[normal,midway, yshift = -0.25cm] {$b_2$}
    (b1) -- ++(1,1) node {}
    (b1) -- ++(1,-1) node {}
    (b1) -- ++(2,0) node {}
    (b0) -- ++(-1,1) node {}
    (b0) -- ++(-1,-1) node {}
    (v0) -- ++(0,-1.5) node {};
    \draw (uf) -- (ue);
  \end{scope}
\end{scope}
\end{scope}
\end{tikzpicture}
  \caption{The sets $C$, $D$, $M$, $C_A$, $M_{A}$, $D_A$, $C_B$, $M_{B}$,
    and $D_B$ in the original decomposition (top) and in the
    segregation that we consider in the proof of \autoref{cl:sep}
    (bottom).}
  \label{imageless}
\end{figure}
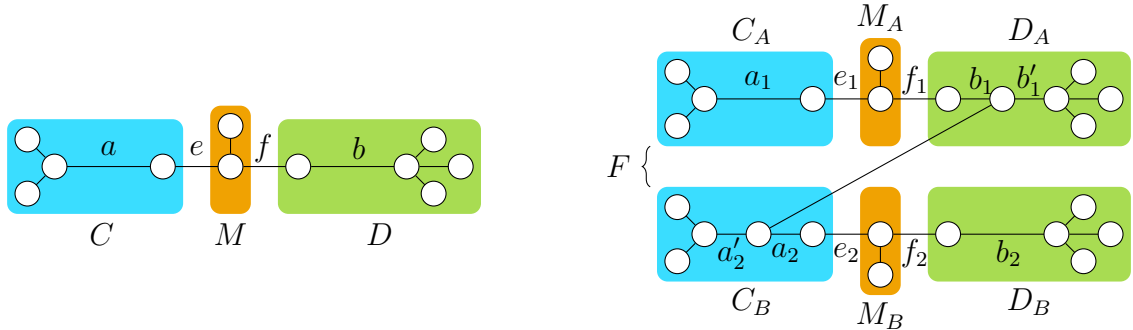
Let us call $T_C$, $T_M$ and $T_D$ the connected components of $T
- \{e,f\}$ that contain, respectively, one endpoint of $e$ and
none of $f$, both one endpoint of $e$ one of $f$, and one endpoint of
$f$ but none of $e$.
As in the proof of \autoref{recurrent}, we set $C = \bigcup_{t \in
  V(T_C)} X_t$, $M = \bigcup_{t \in V(T_M)} X_t$, and $D =
\bigcup_{t \in V(T_D)} X_t$, and for every $i \in [2]$ we define
\begin{eqnarray*}
C_A = C \cap V(G_A) &\quad M_A = M \cap V(G_A)&\quad D_A = D \cap V(G_A), \\
C_B = C \cap V(G_B) &\quad M_B = M \cap V(G_B)&\quad D_B = D \cap V(G_B). \label{achieving}
\end{eqnarray*}
These sets are depicted in \autoref{imageless} on an example of a tree-cut decomposition.
Notice that $F$ contains all edges that have the one endpoint in $C_A \cup M_{A} \cup D_A$ and the other in $C_B \cup M_{B} \cup D_B$. In other words:
\begin{align}
  F  = & E(C_{A},C_{B})\nonumber\\
       & \cup E(C_{A},M_{B}\cup D_{B})\nonumber\\
       & \cup E(C_{B},M_{A}\cup D_{A})\nonumber\\
       & \cup E(M_{A}\cup D_{A},M_{B}\cup D_{B})\label{modifying}
\end{align}
and
\begin{align}
  F  = & E(D_{A},D_{B})\nonumber\\
       & \cup E(D_{A},M_{B}\cup C_{B})\nonumber\\
       & \cup E(D_{B},M_{A}\cup C_{A})\nonumber\\
       & \cup E(M_{A}\cup C_{A},M_{B}\cup C_{B}).\label{directors}
\end{align}
On the other hand, we have
\begin{eqnarray}
  E(C_B, M_{B} \cup D_B) &  \subseteq  & \adh(e_{2}), \label{interplay}\\
 E(D_A, M_{A} \cup C_A) & \subseteq & \adh(f_1), \label{paralysis}
 \end{eqnarray}
 \noindent and
 \begin{align}
   \adh(e) = & E(C_A, M_{A} \cup D_A) \cup E(C_A, M_{B} \cup D_B) \nonumber\\
             & \cup E(C_B, M_{A}\cup D_A) \cup E(C_B, M_{B} \cup D_B),\label{kidnapped}\\
   \adh(f) = & E(D_A, M_{A} \cup C_A) \cup E(D_A, M_{B} \cup C_B) \nonumber\\
             & \cup E(D_B, M_{A} \cup C_A) \cup E(D_B, M_{B} \cup C_B).\label{confessed}
\end{align}
From~\eqref{interplay},~\eqref{paralysis}, and~\eqref{dialectic} we have
\begin{eqnarray}
E(C_B, M_{B} \cup D_B) &  \subseteq & F, \label{automaton}\\
E(D_A, M_{A} \cup C_A) &  \subseteq & F. \label{retreated}
\end{eqnarray}
Using~\eqref{modifying},~\eqref{kidnapped}, and~\eqref{automaton}
and also~\eqref{directors},~\eqref{confessed}~and \eqref{retreated}, we deduce:
\begin{align}
  \adh(e) \setminus F =& E(C_A, M_{A} \cup D_A),\quad \text{and}\label{overflows}\\
  \adh(f) \setminus F =& E(D_B, M_{B} \cup C_B).\label{labourers}
\end{align}
In order to prove that $F$ separates $\adh(e)$ from   $\adh(f)$, let $P$ be a path in $G$ connecting an edge of $\adh(e)$ and and edge of  $\adh(f)$.
If this path contains an edge of $F$, then we are done.
Otherwise, from~\eqref{overflows} and~\eqref{labourers}, $P$ should be
a path from an edge from $E(C_A, M_{A} \cup D_A) \subseteq E(G_A)$ to
an edge from $E(D_B, M_{B} \cup C_B) \subseteq E(G_B)$. Clearly, this path will have an edge in~$F$ 
and the claim follows.
\renewcommand{\qedsymbol}{$\diamond$}
\end{proof}
\renewcommand{\qedsymbol}{$\square$}

Let $e,f\in E(T)$ be links satisfying the conditions of~\autoref{precursor}. 
As $(k, a, b, A, B)$ is a non-leanness certificate, $|\adh(e)|,|\adh(f)|>k$
so we can define $W_e$ as a $k$-element subset of $\adh(e)$ and similarily for $W_f$ in $\adh(f)$.
We now claim that the quintuple $(k,e,f,W_e,W_f)$ is a
{non-leanness certificate} for $(T,\X)$. Condition \itemref{distanced}
follows as, from \autoref{cl:sep}, $F$ separates $\adh(e)$ from $\adh(f)$.
Condition~\itemref{commanded} holds because $eTf$ is a
subpath of $aTb$.
Notice now that $e$ and $f$ are incident while $a$ and $b$ are not. Therefore, the distance between $e$ and $f$  in $T$ is smaller than that between $a$ and $b$. This contradicts the minimality of the choice of 
 $(k,a,b,A,B)$ as a minimal {non-leanness certificate} for $(T,\X)$.  
 \autoref{classical}~follows.
\renewcommand{\qedsymbol}{$\blacksquare$}
\end{proof}
\renewcommand{\qedsymbol}{$\square$}

\begin{sublemma}\label{overgrown}
If $T$ does not contain any $k$-excessive link,  then there is an integer $l \geq k$
  such that 
 \begin{itemize} 
\item   for
  every $j \in [l+1, m]$, $\alpha'_j = \alpha_j$ and $\beta'_j \geq \beta_j$ , while 
  \item $\alpha'_l \leq \alpha_l$ and $\beta'_l > \beta_l$.
  \end{itemize}
\end{sublemma}

\begin{proof}[Proof of \autoref{overgrown}]
  Let us assume that $T$ has no $k$-excessive link. By
  \autoref{classical}, $a$ and $b$ are either incident edges, or they
  are the same edge.
  We set $l= \min \{|\adh(a)|,|\adh(b)|\}$ and observe that $l\geq k$. Clearly, $l$ satisfies the requirements of
  \autoref{committal}, so we have $\alpha'_j = \alpha_j$ and
  $\beta'_j \geq \beta_j$ for every $j \in [l+1, m]$ and
  $\alpha'_l \leq \alpha_l$.
  Let $\psi$ be the function that maps connected components of
  $U^{\geq l}$ to connected components of $T^{\geq l}$, as defined in
  the proof of \autoref{committal} (for $j=l$), where it is shown to be surjective.
  
  As $a$ is not $k$-excessive, the first statement of~\autoref{withstood} does not hold for $e=a$.
  We conclude that 
  for some $i \in [2]$, $|\adh(a)| =
  |\adh(a_i)|$  and $\adh(a_{3-i}) \subseteq F$.
  The case $i=2$ is not possible because then $\adh(a_{1})\subseteq
  F$ which, together with $A \subseteq \adh(a_1) \cup F$ from
  \autoref{vegetable}, implies $k\leq |\adh(a_{1})| \leq |F|<k$, a
  contradiction. Hence $i=1$ and we have that $|\adh(a)| =
  |\adh(a_1)|$. By definition $l \leq |\adh(a)|$, hence $a_1$
  belongs to $U^{\ge l}$. Symmetrically, we can show that $b_2$
  belongs to $U^{\ge l}$.

  From the definition of $l$, both $a$ and $b$ belong
  to $T^{\geq l}$. As they are incident or equal we get that they
  belong to the same connected component of this graph.
  Besides, as noted above, both $a_1$ and $b_2$ belong to $U^{\geq l}$. However
  these links are separated in $U$ by the link $s_1s_2$, which has
  adhesion $|F| < k \leq l$. (Recall that $s_1s_2$ is the link added
  in the construction of $U$ to join the two copies $U_1$ and $U_2$ of
  $T$; see
  \autoref{possesses} for a reminder.)
  Therefore, $a_1$ and $b_2$ do not belong to the same connected
  component of $U^{\ge l}$.
  This proves that $\psi$ is not
  injective: $U^{\ge l}$ has more connected components than $T^{\ge l}$. Therefore,
  $\beta'_l > \beta_l$.
  \renewcommand{\qedsymbol}{$\blacksquare$}
\end{proof}
\renewcommand{\qedsymbol}{$\square$}
\medskip

We are now in position to conclude the proof of~\autoref{dialogues}.

\begin{sublemma}\label{slem:fatfat}
  The fatness of  $(U, \mathcal{Y})$ is smaller than that of~$(T, \X)$.
\end{sublemma}
\begin{proof}
  Recall that we respectively denote by
\begin{gather*}
  (\alpha_m, -\beta_m, \alpha_{m-1},
  -\beta_{m-1}, \dots, \alpha_1, -\beta_1)\\
  \text{and}\\
  (\alpha'_m, -\beta'_m,
 \alpha'_{m-1}, -\beta'_{m-1}, \dots, \alpha'_1, -\beta'_1)
\end{gather*}
the fatnesses of $(T, \X)$ and $(U, \mathcal{Y})$.
Notice that if the assumption of~\autoref{reexamine} or of~\autoref{overgrown}
holds, then the fatness of $(U, \mathcal{Y})$ is (strictly) smaller
than that of $(T, \X)$.
As the assumptions of \autoref{reexamine} and~\autoref{overgrown}
are complementary, we are done.
\end{proof}

Sublemmas~\ref{slem:widthU} and \ref{slem:fatfat} show that $(U, \mathcal{Y})$ has the desired properties, so we are done.
\end{proof}

\begin{lemma}\label{princeton}
  Every 3-edge connected graph $G$ has a lean tree-cut decomposition of width~$\tcw(G)$.
\end{lemma}

\begin{proof}
  Recall that we order fatnesses by lexicographic order.
  The lemma follows from~\autoref{dialogues} and the fact that
  the set of fatnesses of tree-cut decompositions of a given graph does
  not contain an infinite decreasing sequence.
\end{proof}

Based on~\autoref{princeton}, we are now ready to give the proof
of~\autoref{pluralism}. The proof is essentially a reduction of the
general case to that of 3-edge-connected graphs, that is handled
by~\autoref{princeton}.

\begin{proof}[Proof of~\autoref{pluralism}]
  For every $w \in \N$, we show that every graph $G$ such that
  $\tcw(G) \leq w$ has a tree-cut decomposition of width at most $w$
  that is lean, by induction on the number of vertices of $G$.

  Let $w \in \N$ and let $G$ be a graph of tree-cut width at most~$w$.
  If $|V(G)|\leq w$, then the tree-cut decomposition $((\{t\},\emptyset),\{V(G)\})$ 
  has width at most $w$ and is trivially lean.
  Suppose now that $|V(G)| > w$ and that the statement holds for all
  graphs with less vertices than~$G$ (induction hypothesis).

Let $F$ be a cut of $G$ of minimum order and let $\{V_1,V_2\}$ be the
corresponding partition of~$V(G)$. (As we allow multiedges, it is possible
that the edges in $F$ share both endpoints.)
If $|F|>2$ then $G$ is 3-edge-connected and the result follows
because of \autoref{princeton}. Suppose now  that $|F|\leq 2$.
For every $i \in [2]$ we define $G_i$ as follows.
If $|F| = 2$ and the endpoints of $F$ in $V_i$ are distinct, we denote
by $G_i$ the graph obtained from $G[V_i]$ by adding an edge between
these endpoints (or increasing the
multiplicity by one if the edge already exists).
In all the other cases, we set $G_i = G[V_i]$.
Notice that $G_1$ and $G_2$ are both immersions of $G$, hence they
have tree-cut-width at most~$w$ (see Lemma~10 of \cite{Wollan201547}). Also, they have less vertices than
$G$, so we can apply our induction hypothesis.

For every $i\in[2]$, let $(T^i, \X^i)$ be a tree-cut decomposition of
width $\tcw(G_i)\leq w$ of $G_i$, that is lean. Let $x_i$ be an
endpoint of $F$ in $V_i$ or, in the case $F = \emptyset$, any vertex
of $V_i$.
Let $t_i$ be the node of $T^i$ such that $x_i \in X^i_{t_i}$.
We define $T$ as the tree obtained from the disjoint union of $T^1$
and $T^2$ by adding the link $t_1t_2$. We also set $\X = \X^1 \cup
X^2$. Clearly $(T, \X)$ is a tree-cut decomposition of $G$.
Notice that the adhesion of $t_1t_2$ is $F$ whose size is at most~2, hence it is not
bold.
Also, for every $i \in [2]$ and $e \in T^i$, if $\adh_{T^i}(e)$ contains
the edge added in the construction of $G_i$ (if any), then
$\adh_{T}(e)$ contains instead one of the edges of~$F$. Hence the size
of the adhesion of $e$ does not change from $T^i$ to~$T$.
We can therefore express the width of $(T,\X)$ in terms of $(T^i,
\X^i)$ and $F$:
\begin{align*}
  \width(T, \X) &= \max \left \{ \max_{e \in E(T)} |\adh_{(T, \X)}(e)|,\right .\\
                &\qquad \qquad \left .\max_{t \in V(T)} \left (|X_t| + |\{ t' \in N_T(t),\ \adh_{(T, \X)}(tt')\
                  \text{is bold} \}| \right ) \right \}\\
                &= \max \left \{ \max_{e \in E(T^1)} |\adh_{(T^1, \X^1)}(e)|,\
                  \max_{e \in E(T^2)} |\adh_{(T^2, \X^2)}(e)|,\ |F|,\right .\\
                &\qquad \qquad \max_{t \in V(T^1)} \left (|X_t| + |\{ t' \in N_{T^1}(t),\ \adh_{(T^1, \X^1)}(tt')\
                  \text{is bold} \}| \right ),\\
                &\qquad \qquad \left. \max_{t \in V(T^2)} \left (|X_t| + |\{ t' \in N_{T^2}(t),\ \adh_{(T^2, \X^2)}(tt')\
                  \text{is bold} \}| \right ) \right \}\\
                 &= \max\left \{ \width(T^1, \X^1),\ \width(T^2, X^2),
                   |F|\right\}\\
                &\leq \max\{w, |F|\}
\end{align*}
As $|V(G)| > w \geq \tcw(G)$, the tree of any miminum-width tree-cut
decomposition of $G$ has at least one link and the adhesion of this
link has size at most $w$. Hence $G$ has a cut of size at most
$w$. By minimality of $F$, we deduce $|F|\leq w$, hence $\width(T, \X)
\leq w$ from the inequalities above.

At this point of the proof we have constructed a tree-cut decomposition
$(T, \X)$ of $G$ of width at most~$w$. It remains to show that it is lean.
For this, we consider some $a,b \in E(T)$ and subsets $A \subseteq
\adh(b)$ and $B \subseteq \adh(b)$ of the same size~$k$.
It is enough in order to conclude the proof to assume
that there is no collection of $k$ edge-disjoint paths linking $A$ to
$B$ in~$G$ and show that $aTb$ has a link whose adhesion has size
less than $k$.

In the case where $a,b \in E(T^i)$ for some $i\in[2]$, we observe
that, as it is an immersion of $G$, $G_i$ does not contain $k$ edge-disjoint
paths linking $A$ to~$B$.
Because $(T^i, \X^i)$ is lean, there is a link $e$ in $aT^ib$
such that $|\adh_{(T^i, \X^i)}(e)|< k$ (in particular, $a \neq b$). As
noted above, $\left |\adh_{(T^i, \X^i)}(e) \right | = \left |\adh_{(T,
    \X)}(e)\right |$ for every $e \in E(T^i)$, hence $|\adh_{(T, \X)}(e)|< k$ and we are done.
It remains to consider the case where $a$ and $b$ do not belong to the
same of $T^1$ and $T^2$.
By \autoref{testament} (the variant of Menger's Theorem), there is in $G$ a cut of size strictly smaller
than $k$ that separates $A$ from $B$.
By minimality of $F$, this implies $|F| < k$.
Observe then that $t_1t_2$ is an edge of $aTb$
and $|\adh(t_1t_2)| = |F| <k$, as desired.
\end{proof}

\section*{Acknowledgments}

The authors wish to thank Micha\l{} Pilipczuk and Marcin Wrochna for
valuable discussions about lean tree-cut decompositions.
%
\newcommand{\etalchar}[1]{$^{#1}$}

\end{document}